\documentclass{amsart}
\usepackage{amsmath}
\usepackage{amsthm}
\usepackage[foot]{amsaddr}
\usepackage{amsfonts,mathrsfs}
\usepackage[flushleft]{threeparttable}
\usepackage{amssymb}
\usepackage{graphicx}
\usepackage{enumitem}
\usepackage{color}
\usepackage{verbatim}
\theoremstyle{plain}
\newtheorem{theorem}{Theorem}[section]
\newtheorem{lemma}[theorem]{Lemma}
\newtheorem{example}[theorem]{Example}
\newtheorem{corollary}[theorem]{Corollary}
\newtheorem{proposition}[theorem]{Proposition}

\theoremstyle{definition}

\newtheoremstyle{TheoremNum}
	{\topsep}{\topsep}              
  {\itshape}                      
  {}                              
  {\bfseries}                     
  {.}                             
  { }                             
  {\thmname{#1}\thmnote{ \bfseries #3}}
\newtheorem{remark}[theorem]{Remark}

\newcommand{\F}{\mathbb F}
\newcommand{\K}{\mathbb K}
\newcommand{\Z}{\mathbb Z}
\newcommand{\C}{\mathcal C}
\newcommand{\cD}{\mathcal D}
\newcommand{\bD}{\mathbb D}

\newcommand{\cS}{\mathcal S}

\newcommand{\bv}{\mathbf v}
\newcommand{\bili}{\mathtt b}
\newcommand{\bx}{\mathbf x}
\newcommand{\by}{\mathbf y}

\newcommand{\ba}{\mathbf a}
\newcommand{\bb}{\mathbf b}

\newcommand{\cC}{\mathscr C}
\newcommand{\cH}{\mathcal H}
\newcommand{\cG}{\mathcal G}
\newcommand{\id}{\mathrm{id}}

\newcommand{\T}{\mathbb T}
\newcommand{\scrT}{\mathscr T}

\newcommand{\bbS}{\mathbb S}

\newcommand{\Aut}{\mathrm{Aut}}
\newcommand{\Hom}{\mathrm{Hom}}
\newcommand{\End}{\mathrm{End}}
\newcommand{\GL}{\mathrm{GL}}

\newcommand{\rk}{\mathrm{rk}}
\newcommand{\Tr}{ \ensuremath{ \mathrm{Tr}}}

\newcommand{\RN}[1]{%
  \textup{\uppercase\expandafter{\romannumeral#1}}%
}
\newcommand{\rn}[1]{%
  \textup{\lowercase\expandafter{\romannumeral#1}}%
}

 \def\zhou#1 {\fbox {\footnote {\ }}\ \footnotetext { From Yue: {\color{red}#1}}}

 \def\rocco#1 {\fbox {\footnote {\ }}\ \footnotetext { From Rocco: {\color{blue}#1}}}

\begin{document}
	\title{On kernels and nuclei of rank metric codes}
	\author[G. Lunardon]{Guglielmo Lunardon\textsuperscript{\,1}}
	\author[R. Trombetti]{Rocco Trombetti\textsuperscript{\,1}}
	\address{\textsuperscript{1}Dipartimento di Mathematica e Applicazioni ``R. Caccioppoli", Universit\`{a} degli Studi di Napoli ``Federico \RN{2}", I-80126 Napoli, Italy}
	\email{rtrombet@unina.it}
	\author[Y. Zhou]{Yue Zhou\textsuperscript{\,2,3}}
	\address{\textsuperscript{2}College of Science, National University of Defense Technology, 410073 Changsha, China}
	\address{\textsuperscript{3}Department of Mathematics, University of Augsburg, 86135 Augsburg, Germany}
	\email{yue.zhou.ovgu@gmail.com}
	\date{\today}
	\maketitle
	
	\begin{abstract}
	For each rank metric code $\C\subseteq \K^{m\times n}$, we associate a translation structure, the kernel of which is shown to be invariant with respect to the equivalence on rank metric codes. When $\C$ is $\K$-linear, we also propose and investigate other two invariants called its middle nucleus and right nucleus. When $\K$ is a finite field $\F_q$ and $\C$ is a maximum rank distance code with minimum distance $d<\min\{m,n\}$ or $\gcd(m,n)=1$, the kernel of the associated translation structure is proved to be $\F_q$. Furthermore, we also show that the middle nucleus of a linear maximum rank distance code over $\F_q$ must be a finite field; its right nucleus also has to be a finite field under the condition $\max\{d,m-d+2\} \geqslant \left\lfloor \frac{n}{2} \right\rfloor +1$. Let $\cD$ be the DHO-set associated with a bilinear dimensional dual hyperoval over $\F_2$. The set $\cD$ gives rise to a linear rank metric code, and  we show that its kernel and right nucleus are is isomorphic to $\F_2$. Also, its middle nucleus must be a finite field containing $\F_q$. Moreover, we also consider the kernel and the nuclei of $\cD^k$ where $k$ is a Knuth operation. 
	\end{abstract}

\section{Introduction}
Let $\K$ be a field. The set $\K^{m\times n}$ of all $m\times n$ matrices over $\K$ is a $\K$-vector space. The \emph{rank metric distance} on the $\K^{m\times n}$ is defined by 
\[d(A,B)=\mathrm{rk}(A-B) \,\, \text{for} \,\, A,B\in \K^{m\times n},\]
where $\rk(C)$ stands for the rank of $C$.

A subset $\C\subseteq \K^{m\times n}$ is called a \emph{rank metric code}. The \emph{minimum distance} of $\C$ is
\[d(\C)=\min_{A,B\in \C, A\neq B} \{d(A,B)\}.\]
When $\C$ is a $\K$-linear subspace of $\K^{m\times n}$, we say that $\C$ is a $\K$-linear code and its dimension $\dim_{\K}(\C)$ is defined to be the dimension of $\C$ as a subspace over $\K$.

There are several interesting structures in finite geometry, cryptography and coding theory, which can be equivalently described in the context of rank metric codes. First, a quasifield is an algebraic structure with two binary operations which are often called its addition and multiplication. Quasifields are quite similar to skewfields, but with some weaker conditions. Quasifields of finite order are strongly related to translation planes in finite geometry. A quasifield of order $q^n$ with kernel $\F_q$ can be viewed as a subset $\C$ of $q^n$ matrices in $\F_q^{n\times n}$ satisfying that the zero matrix is in $\C$ and $d(\C)=n$. This subset $\C$ is often called a \emph{spreadset}. In particular, when $\C$ is $\F_q$-linear, it defines a finite \emph{semifield}, which is a quasifield with two-sided distributivity. For more details on quasifields and semifields, we refer to \cite{hughes_projective_1973,johnson_handbook_2007,lavrauw_semifields_2011}.

Another interesting topic is from cryptography and coding theory: A function $f: \F_{2^n}\rightarrow \F_{2^m}$ is called \emph{almost perfect nonlinear} (abbreviated to APN), if $\#\{x: f(x+a)+f(x)=b\}=0$ or $2$ for all $a\in \F_{2^n}^*$ and $b\in \F_{2^m}$. APN functions are of interest in the design of S-boxes, which are basic components of symmetric key algorithms. When $n=m$, except for the six families of APN monomials, most known families of APN functions are \emph{quadratic}, i.e.\ $f(x)=\sum_{i\le j} a_{ij} x^{2^i + 2^j}$. It is easy to see that the map given by $x\mapsto f(x+a)+f(x)+f(a)$ for each nonzero $a$ can be viewed as a matrix $M_a$ of rank $n-1$ in $\F_2^{n\times n}$. Furthermore, all $M_a$ together with the zero matrix form a $\F_2$-linear code $\C$ in $\F_{2}^{n\times n}$ and $d(\C)=n-1$.  We refer to \cite{blondeau_perfect_2015,pott_almost_2016} for  recent surveys on APN functions.

A quadratic APN function can be viewed geometrically as a special type of \emph{dimensional dual hyperoval} (DHO for short). Every known DHO is \emph{splitting}, which means that it can be described as a set $\cD$ of matrices, called a \emph{DHO-set}, in $\F_q^{n\times m}$ for certain $q$, $n$ and $m$. A DHO-set $\cD$ has an important property that the difference of any two distinct matrices in it are of rank $n-1$, whence $\cD$ is also a rank metric code and $d(\cD)=n-1$.

Rank metric codes are also useful in the construction of error correcting codes for random network coding and of some transversal designs \cite{koetter_coding_2008,silva_rank-metric_2008}.

Let $\C\subseteq \F_q^{m\times n}$. When $d(\C)=d$, it is well-known that
$$\#\C\le q^{\max\{m,n\}(\min\{m,n\}-d+1)},$$
which is the Singleton bound for the rank metric distance; see \cite{delsarte_bilinear_1978}. When the equality holds, we call $\C$ a \emph{maximum rank distance} (MRD for short) code. It is clear that the spreadset derived from a quasifield of order $q^n$ is an MRD code in $\F_q^{n\times n}$ and its minimum distance is $n$. For MRD codes with minimum distance less than $\min\{m,n\}$, there are a few known constructions. The first and most famous family is due to Gabidulin \cite{gabidulin_MRD_1985} and Delsarte \cite{delsarte_bilinear_1978} who found it independently. This family is later generalized by Kshevetskiy and Gabidulin in \cite{kshevetskiy_new_2005}, and we often call them \emph{Generalized Gabidulin codes}. Recent constructions of MRD codes can be found in \cite{cossidente_nonlinear_2015,horlemann-trautmann_new_2015,lunardon_generalized_2015,sheekey_new_2015}. Also, in \cite{lunardon_preprint} some relationship between linear MRD codes and different geometric objects like {\it linear sets} of a projective space and {\it generalized Segre varieties} were pointed out.

In general, it is difficult to tell whether two rank metric codes with the same parameters are equivalent or not. For quasifields, in particular for semifields, there are several classical invariants such as kernel, left, right and middle nuclei. Originally they are defined as algebraic substructures of quasifields or semifields. However they can also be translated into the language of matrices. For more information on the nuclei of finite semifields, we refer to \cite{marino_nuclei_2012}. These invariants are quite useful in telling the equivalence between two semifields, and many classification results on semifields are also based on certain assumptions on the sizes of their nuclei; see \cite{marino_nuclei_2012,marino_towards_2011,menichetti_kaplansky_1977,menichetti_n-dimensional_1996} for instance. Hence it is quite natural to ask whether there are also such invariants for other rank metric codes, especially for MRD codes and DHO-sets.

The organization and the main results of this paper are as follows: In Section \ref{se:pre}, we introduce several important concepts including the equivalence on rank metric codes together with translation structures. In Section \ref{se:translation_structures}, we associate with a rank metric code $\C$ a point-line incidence translation structure $\scrT(\C)$, i.e., an incidence structure with an equivalence relation defined on the set of lines and with a group acting sharply transitively on its points. We investigate properties of the kernel $K$ of such an incidence structure. In Section \ref{se:nuclei}, the {\it middle nucleus} and the {\it right nucleus} of a linear rank metric code is introduced and proved to be invariants under codes equivalence . Relations between the middle nucleus and the right one of a rank metric code is investigated. In Section \ref{se:MRD}, we look at the kernel and the nuclei of an MRD code $\C\subseteq\F_q^{m\times n}$. We show that its kernel is
 
$\F_q$ under the condition that its minimum distance $d<\min \{m,n\}$ or $\gcd(m,n)=1$. Moreover, we also prove that the middle nucleus of $\C$ is always a finite field and its right nucleus is a finite field if $\max\{d,m-d+2\} \geqslant \left\lfloor \frac{n}{2} \right\rfloor +1$. For the case $m=n$, we determine the middle (right) nuclei of generalized (twisted) Gabidulin codes. 

In Section \ref{se:DHO}, we introduce dimensional dual hyperovals and associated DHO-sets. We deal with some related concepts  as well as the opposite operation $\circ$ and the adjoint operation $\dagger$ defined on a DHO-set. We observe that, by choosing an appropriate bases, this latter operation gives rise to the adjoint code $\cD^{\top}$ of $\cD$. We completely determine the kernel of the translation structure derived from an arbitrary DHO. Finally, we concentrate on splitting bilinear DHOs $\bD$. For the DHO-set $\cD$ associated with such a $\bD$,  we determine the middle (right) nuclei of $\cD^k$ for $k\in \{\circ, \top, \circ\top, \top\circ, \top\circ\top  \}$.

\section{Preliminaries}\label{se:pre}
In this section, we introduce several important concepts and results on rank metric codes and basic facts on translation structures.

First, let us fix several notations. For any matrix $M$, we use $M^t$ to denote the transpose of $M$ and $\rk(M)$ is the rank of $M$. We also use ${O}_{m,n}$ to denote an $m\times n$ zero matrix over a field. If the numbers of rows and columns are clear from the context, we simply write it as $O$. We always use Latin letters in bold, such as $\bx,\by, {\bf z}$ to represent (row) vectors.

Let $\C$ be a rank metric code in $\K^{m\times n}$.  The \emph{adjoint code} of $\C$ is the code 
\[\C^{\top}:=\{X^t : X \in \C\}.\]

Let $\langle \cdot, \cdot \rangle$ be the symmetric bilinear form on the set of $m\times n$ matrices defined by
\[\langle M,N\rangle:= \Tr(MN^t).\]
 The \emph{Delsarte dual code} of a $\K$-linear code $\C$ is 
\[\C^\perp :=\{M\in \K^{m\times n}:\langle M,N \rangle=0\text{ for all } N\in \C  \}.\]

One important result proved by Delsarte \cite{delsarte_bilinear_1978} is that the Delsarte dual code of a linear MRD code is still MRD. Also, if $d>1$, then
\begin{equation}\label{eq:delsarte_dual_distance}
	d({\C}^{\perp})=\text{min}\{m,n\}-d+2.
\end{equation}
For the trivial case $d=1$, $\C=\K^{m\times n}$ and $\C^\perp$ consists of a zero matrix.

For any matrix $M$ over a field $\K$ and $\gamma\in \Aut(\K)$, we define $M^{\gamma}=(m_{ij}^{\gamma})$.

Let $m,n$ be two integers larger than $1$. An \emph{isometry} on $\K^{m\times n}$ is a bijection which preserves the rank distance. In \cite[Theorem 3.4]{wan_geometry_1996}, it is proved that if $\varphi$ is an isometry on $\K^{m\times n}$, then there are $A\in \GL(m,\K)$, $B\in \GL(n,\K)$, $C\in \K^{m\times n}$ and $\gamma \in \Aut(\K)$ such that 
\begin{equation}\label{eq:isometry_cond_1}
	\varphi(X)=A X^\gamma B+C
\end{equation}
for all $X\in \K^{m\times n}$, or (when $m=n$) 
\begin{equation}\label{eq:isometry_cond_2}
	\varphi(X)=A (X^t)^\gamma B+C
\end{equation}
for all $X\in \K^{m\times n}$. 

As the isometries on $\K^{m\times n}$ keep the rank distance, following the definition in \cite{de_la_cruz_algebraic_2015} we should use isometry as the equivalence on rank metric codes. However, for convenience, we use the following two definitions in this paper. Two rank metric codes $\C_1$ and $\C_2\subseteq \K^{m\times n}$ are \textit{equivalent} if there are $A\in \GL(m,\K)$, $B\in \GL(n,\K)$, $C\in \K^{m\times n}$ and $\gamma \in \Aut(\K)$ such that 
\begin{equation}\label{eq:equivalence_def}
	\C_2=\{AX^{\gamma}B + C : X \in \C_1\}.
\end{equation}  
When $m=n$, we say that $\C_1$ and $\C_2$ are \textit{strongly equivalent} if $\C_2$ is equivalent either to $\C_1$ or to $\C^{\top}_1$. Therefore, if $m\neq n$, isometry and equivalence are the same; otherwise $m=n$, isometry is the same as strong equivalence. 

An equivalence map from a rank metric code $\C$ to itself is called an \emph{automorphism}. All automorphisms together form the \emph{automorphism group} of $\C$.

When $\C_1$ and $\C_2$ are linear, by letting $X=O$ in \eqref{eq:equivalence_def} we see that $C\in \C_2$ and $\C_2-C:=\{Y-C : Y\in \C_2  \}=\C_2$, which means that we may always assume that $C=O$.

The first example of a linear MRD code of $m\times n$ matrices existing for arbitrary value of the minimum distance $d$, was exhibited by Delsarte in \cite{delsarte_bilinear_1978} and independently by Gabidulin in \cite{gabidulin_MRD_1985}, and it was later generalized by Kshevetskiy and Gabidulin in \cite{kshevetskiy_new_2005}. We often call them (generalized) Gabidulin codes.

Precisely, a generalized Gabidulin code is defined as follows: It is well-known that, under a given basis of $\F_{q^n}$ over $\F_q$,  each element $a$ of $\F_{q^n}$ can be written as a (column) vector $\bv(a)$ in $\F_{q}^n$. Let $\alpha_1,\dots,\alpha_m$ be a set of linear independent elements of $\F_{q^n}$ over $\F_q$, where $m\le n$. Then
\begin{equation}\label{eq:mn_MRD}
\left\{ \left(\bv(f(\alpha_1)), \dots, \bv(f(\alpha_m))\right)^t: f\in \cG_{k,s}
  \right\}
\end{equation}
is the original generalized Gabidulin code, where 
\begin{equation}\label{eq:GG}
	\cG_{k,s} = \{a_0 x + a_1 x^{q^{s}} + \dots +a_{k-1} x^{q^{s(k-1)}}: a_0,a_1,\dots, a_{k-1}\in \F_{q^n} \},
\end{equation} with $n,k,s\in \Z^+$ satisfying $k<n$ and $\gcd(n,s)=1$.
To get the minimum distance of this code, we only have to look at the number of the roots of each $f\in \cG_{k,s}$.

All members of $\cG_{k,s}$ are of the form $f(x)= \sum_{i=0}^{n-1}a_i x^{q^i}$, where $a_i\in \F_{q^n}$. A polynomial of this form is called a \emph{linearized polynomial} (also a $q$-polynomial because its exponents are all powers of $q$). They are equivalent to $\F_q$-linear transformations from $\F_{q^n}$ to itself, i.e., elements of $\mathbb E=\End_{\F_q}(\F_{q^n})$. We refer to \cite{lidl_finite_1997} for their basic properties.

A \emph{semifield} $\mathbb{S}$ is an algebraic structure satisfying all the axioms of a skewfield except (possibly) the associative law of multiplication. It is not difficult to show that the additive group of a semifield $\mathbb{S}$ is an elementary abelian group; see \cite{knuth_finite_1965}. The additive order of the nonzero elements in $\mathbb{S}$ is called the characteristic of $\mathbb{S}$. Hence, any finite semifield can be represented by $(\mathbb{F}_q, +, *)$ with a prime power $q$. Here $(\mathbb{F}_q, +)$ is the additive group of the finite field $\mathbb{F}_q$ and $x*y=\omega(x,y)$, where $\omega$ is a mapping from $\mathbb{F}_q\times \mathbb{F}_q$ to $\mathbb{F}_q$ satisfying that
\begin{align*}
	(x+y)*z &= x*z + y*z,\\
	x*(y+z) &= x*y + x*z
\end{align*}
for all $x,y,z\in \F_{q}$. That means the map $x\mapsto x*y$ as well as $x\mapsto y*x$ also give rise to two linearized polynomials over a certain subfield of $\F_q$. By definition, these two maps must be invertible for $y\ne 0$. Hence, from them we can derive two MRD codes consisting of $q-1$ nondegenerate matrices with the zero matrix. For instance, if we take the finite field $\F_{p^n}$ which is obviously a semifield, then we can get a set of $p^n$ matrices in $\F_p^{n\times n}$ defined by the (left, right) multiplication in $\F_{p^n}$.

The left, middle and right nucleus of a semifield $\bbS$ are the following subsets:
\begin{align*}
  N_l(\mathbb{S})=\{a\in \mathbb{S}: (a*x)*y=a*(x*y) \text{ for all }x,y\in \mathbb{S}\},\\
  N_m(\mathbb{S})=\{a\in \mathbb{S}: (x*a)*y=x*(a*y) \text{ for all }x,y\in \mathbb{S}\},\\
  N_r(\mathbb{S})=\{a\in \mathbb{S}: (x*y)*a=x*(y*a) \text{ for all }x,y\in \mathbb{S}\}.
\end{align*}

For a rank metric code $\C\in \K^{m\times n}$ provided that $\C$ is finite, the \emph{rank weight distribution} of $\C$ is a sequence of numbers 
\[A_j:=\# \{M: M\in\C,  \rk(M)=j \}\] 
for $j=0,1,\dots, \min\{m,n\}$. In general, it is difficult to determine the rank weight distribution of a given code. However, MRD codes with the same parameters have the same rank weight distribution which is completely known. Without loss of generality, we assume that $n\geqslant m$ and $\C$ is an MRD code in $\F_q^{m\times n}$ with minimum distance $d$. Of course $A_j=0$ for $j< d$. In \cite{delsarte_bilinear_1978,gabidulin_MRD_1985}, it is proved that
\begin{equation}\label{eq:MRD_weightdistribution}
	A_{d+\ell}= {m \brack d+\ell}_q \sum_{t=0}^\ell(-1)^{t-\ell} {\ell+d \brack \ell-t}_q q^{\binom{\ell-t}{2}} \left(q^{n(t+1)}-1\right),
\end{equation}
for $\ell=0,1,\dots, n-d$, where ${m \brack j}_q$ is the Gaussian binomial coefficient.  In fact, we can prove the following result without doing complicated calculation of \eqref{eq:MRD_weightdistribution}.
\begin{lemma}\label{lm:MRD_weight_all}
	Let $\C$ be an MRD code in $\F_q^{m\times n}$ with minimum distance $d$. Assume that ${O}\in \C$. For any $0\leqslant \ell \leqslant m-d$, we have $A_{d+\ell}>0$, i.e.\ there always exists at least one matrix $C\in \C$ such that $\rk(C)=d+\ell$.
\end{lemma}
\begin{proof}
	As all MRD codes with the same parameters have the same rank distribution,
	we only have to look at the code defined by \eqref{eq:mn_MRD}. Let us denote this code by $\C_k$ where $k= m-d+1$.
	 
	Clearly, for $k=1$, all matrices in $\C_1$ are of full rank. Assume that our lemma holds for $\C_{k_0}$. As $\cG_{k_0,s}\subseteq\cG_{k_0+1,s}$, there exists matrix of rank $r$ in $\C_{k_0+1}$ for $r=m, m-1, \dots, m-k_0+1$. On the other hand, $\C_{k_0+1}$ is an MRD code which means that there must be matrices of rank $m-k_0$ in it. Hence the lemma also holds for $\C_{k_0+1}$. By induction, we complete the proof.
\end{proof}

Finally we turn to the introduction of a particular incidence structure which is called a translation structure.

Let $\mathcal{P}$ be a nonempty set, whose elements are called \emph{points}, and let $\mathcal L$ be a family of subsets of $\mathcal P$, whose elements are called \emph{lines}  or \emph{blocks}. The pair $(\mathcal P, \mathcal L)$ forms an \emph{incidence structure}. A permutation on $\mathcal P$ is called a \emph{collineation} of the incidence structure $(\mathcal P, \mathcal L)$, if it is also a permutation on $\mathcal L$ and preserves the incidence relation.

An incidence structure $\T=({\mathcal P},{\mathcal L})$ with {\it parallelism} is a point-line geometry
endowed with an equivalence relation defined on the set $\mathcal L$ of lines. We denote this relation with the symbol $||$. A {\it translation} of $\T$ is a collineation $\tau$ such that $L^{\tau}||L$
for all lines $L$ of $\T.$ The translations of $\T$ form a group $T$.
We call $(\T,T)$ a {\em translation structure} if
\begin{enumerate}[label=(\alph*)]
\item	the group $T$ acts sharply transitively on the points of $\T$;
\item if $L$ is a line of $\T,$ then the stabilizer $T_L$ of $L$ in $T$ is transitive on the points of $L$.
\end{enumerate}

The group $T$ is called the \emph{translation group} of $\T$. We say that $\T$ is a \emph{central translation structure} when $T$ is abelian. Two translation structures $\T_1$ and $\T_2$ are said to be \emph{isomorphic} if they are isomorphic as incidence structures, i.e., there is a one-to-one map $\sigma$ from the points (lines) of $\T_1$ to the points (lines) of $\T_2$ such that a point $x$ is in a line $L$ if and only if $\sigma(x) $ is in $\sigma(L)$.

Translation planes are classical examples of a translation structure in which two points are incident with a unique line. Translation structures were introduced by Andr\'{e} in \cite{andre_uber_1961}; see \cite{bader_desarguesian_2010} too. In \cite{andre_uber_1961}, the following canonical representation is given for $(\T,T)$.

Let $x$ be a fixed point of $\T$. For any line $L$ incident with $x$,  define $T_{L}=\{\tau \in T : L^{\tau}=L\}$ and put ${\mathcal S}=\{T_{L} : L \text{ is incident with } x\}.$ 

For each line $M$ of $\T$ there is an element $\tau$ of $T$ and a line $L$ incident with $x$ such that $M=L^{ \tau}.$ Thus the coset $T_{L}\tau$ is the set of the elements of $T$ which map $x$ to a point of $M$ and for each point $y$ of $M$ there is exactly one element $\mu$ of $T_{L} \tau$ such that $x ^{\mu}=y.$

Let $S(T,{\mathcal S})$ be the point-line structure whose points are the elements of $T$ and whose lines are the cosets of elements of ${\mathcal S}.$ For each point $y,$ let $\tau_y$ be the element of $T$ which maps $x$ to $y$ and let $\beta_{x}$ be the map from $\T$ to $S(T, {\mathcal S})$ defined by $y \mapsto \tau_y$ and $M \mapsto T_{L}\tau_{y}$ if and only if $M=L ^{\tau_y}$. Then $\beta_{x}$ is an isomorphism between $\T$ and $S(T,{\mathcal S})$. It is worth noticing that the construction does not depend, up to isomorphism, on the choice of the point $x$. 

We say that the incidence structure $S(T,{\mathcal S})$ satisfies the \emph{covering property}, if
\begin{equation}\label{eq:covering_property}
	\bigcup_{x\in L}T_{L} =T. 
\end{equation} 

The {\it kernel} $K$ of ${\mathcal S}$ is the set of all endomorphisms $\kappa$ of $T$ such that $T_{L}^\kappa \subseteq T_{L}$ for all $L$ incident with $x$. If $T$ is abelian, then $K$ is a  ring (not necessarily commutative) with identity. We will use the exponential notation so that the sum and the multiplication of $K$ are defined by $\tau^{\kappa + \lambda}=\tau^\kappa \tau^\lambda$ and $\tau^{\kappa \lambda}=(\tau^\kappa)^\lambda$ for all $\tau \in T,$ and $\lambda,\kappa \in K$. Then, the group $T$ is a $K-$module and each element of ${\mathcal S}$ is a submodule of $T.$ 




\section{Translation structures from rank metric codes}\label{se:translation_structures}
In this part, we define a translation structure from a set of $m\times n$ matrices. Let $\C$ be a subset of $\K^{m\times n}$ and $\bf0$ denote the zero vector. We define
\begin{align*}
	S(\infty) &:= \{ ({\bf 0}, \by) :  \by \in \K^n  \},\\
	S(M) &:= \{(\bx, \bx M): \bx \in \K^m \}, \text{ for }M\in \C.
\end{align*}

Let $\cS(\C):=\{S(M): M\in \C \cup \{\infty\} \}$. From it we derive an incidence structure  on $\K^{m+n}$, in which the lines are defined by
\begin{align*}
	S(M) + ({\bf 0}, \bf{b }), \quad &\text{for } M\in \C, {\bf b}\in \K^n,\\
	S(\infty) + ({\bf a}, \bf{0 }), \quad &\text{for }{\bf a}\in \K^m. 
\end{align*}
It is routine to verify that this is a translation structure and the additive group of $\K^{m+n}$ is its translation group. 
Let us denote this translation structure by $\scrT(\C)$. 

According to definition, the kernel $K$ of $\scrT(\C)$ is the set of all endomorphisms of the group $(\K^{m+n},+)$ such that $S(M)^\mu \subseteq S(M)$ for every $M\in \C\cup \{\infty\}$. For convenience, we also say that $K$ is the kernel.

\begin{lemma}\label{lm:equivalence_Codes_kernel}
Suppose that $\C_1$ and $\C_2$ are two equivalent rank metric codes in $\K^{m\times n}$. Then the derived translation structures $\scrT(\C_1)$ and $\scrT(\C_2)$ are isomorphic. In particular, their kernels $K_{\C_1}$ and $K_{\C_2}$ are isomorphic.
\end{lemma}
\begin{proof}
Suppose that ${\C}_1$ and ${\C}_2$ are equivalent. By definition we have that ${\C}_2=\{A M^{\sigma} B +C\colon M \in {\C}_1\}$ where $A\in GL(m,\K)$ and $B \in GL(n,\K)$ are nonsingular, $C\in \K^{m\times n}$ and $\sigma \in \Aut(\K)$.  The semilinear map $$\alpha \,:\, ({\bf x},{\bf y}) \in \K^m \times \K^n \mapsto ({\bf x}^{\sigma}A^{-1},{\bf y}^{\sigma}B+{\bf x}^{\sigma}A^{-1}C) \in \K^m \times \K^n,$$ is an isomorphism between $\scrT({\C}_1)$ and $\scrT({\C}_2)$ with $K_{{\C}_2} = \alpha^{-1}K_{{\C}_1}\alpha$.
\end{proof}

By the definition of kernel, the following result is easy to get:
\begin{lemma}\label{lm:K_in_Kernel}
	Let $I_{m+n}$ denote the identity matrix of order $m+n$. The set of matrices $\{aI_{m+n} : a\in \K \}$, which forms a field isomorphic to $\K$, belongs to the kernel $K$ of $\scrT(\C)$.
\end{lemma}

By Lemma \ref{lm:K_in_Kernel}, the field $\K$ is in the kernel $K$ of $\scrT(\C)$. It is interesting and natural to ask whether $K$ is necessarily a field and whether $K$ contains some extra elements. We proceed to investigate these two questions in the rest part of this section.

\begin{lemma}\label{lm:kernel_N1_N2}
	Assume that the zero matrix is in $\C$. Then each element in the kernel $K$ of $\scrT(\C)$ can be expressed in the form
		\[\left(
			  \begin{array}{cc}
			    N_1 & O_{m,n}    \\
			    O_{n,m} & N_2 \\
			  \end{array}
			\right),\]
		where $N_1\in \End((\K^m,+))$, $N_2\in \End((\K^n,+))$ and $O_{m,n}$ (resp.\ $ O_{n,m}$) denotes the zero map in $\Hom((\K^m,+),(\K^n,+))$ (resp. $\Hom((\K^n,+),(\K^m,+))$).
\end{lemma}
\begin{proof}
	Let $\mu$ be an arbitrary element of $K$. As an endomorphism of the additive group of $\K^{m+n}$, $\mu$ can be written as 
	\[\left(
	  \begin{array}{cc}
	    N_1 & N_4 \\
	    N_3 & N_2 \\
	  \end{array}
	\right),\]
	where $N_1\in \End((\K^m,+))$, $N_2\in \End((\K^n,+))$, $N_3\in \Hom((\K^n,+), (\K^m,+))$ and $N_4\in \Hom((\K^m,+), (\K^n,+))$. Note that
	\[S(\infty)^\mu=\{ (\by N_3, \by N_2): \by \in \K^n \}.\]
	Together with $S^\mu(\infty)\subseteq S(\infty)$, we get $\by N_3={\bf 0}$ for every $\by \in \K^n$. Hence $N_3$ is the zero mapping. Similarly we can also show that $N_4=O_{m,n}$ by looking at  $S(O_{m,m})^\mu\subseteq S(O_{m,m})$.	
\end{proof}

\begin{proposition}\label{prop:kernel_adjoint_dual}
	Let $\C$ be a rank metric code containing $\bf 0$.
	\begin{enumerate}[label=(\alph*)]
	\item Let $K$ and $K^\top$ denote the kernels of $\scrT(\C)$ and $\scrT(\C^\top)$ respectively. Then 
	\[K\cap \Aut((\K^{m+n},+))\cong K^\top\cap \Aut((\K^{m+n},+)).\]
	\item Assume that $\C$ is linear. The group of automorphisms of $(\K^{m+n},+)$ stabilizing $\scrT(\C^\perp)$ contains a subgroup which is isomorphic to $K\cap \GL(m+n, \K)$. 
	\end{enumerate}
\end{proposition}
\begin{proof}
	(a). By Lemma \ref{lm:kernel_N1_N2}, we know that an element $\mu$ in $K\cap \Aut(\K^{m+n})$ can be written as
	\[\left(
	  \begin{array}{cc}
	   N_1 & O_{m,n}    \\
	   O_{n,m} & N_2 \\
	  \end{array}
	\right),\]
	where $N_1\in \Aut((\K^m,+))$ and $N_2\in \Aut((\K^n,+))$.
	
	Due to the definition of kernels, for every $M\in \C$ and $\bx\in \K^m$,
	\[(\bx, \bx M)^\mu =(\bx N_1, \bx M N_2)=(\by, \by N_1^{-1} M N_2)=(\by, \by M),\]
	where $\by=\bx N_1$. Hence
	\[N_1^{-1} M N_2=M,\]
	which implies that 
	\begin{equation}\label{eq:N2_N1_Mt}
		N_2^t M^t (N_1^t)^{-1} = M^t.
	\end{equation}
	Hence 
	\[\mu':= \left(
		 \begin{array}{cc}
		  (N_2^t)^{-1} &     \\
		      & (N_1^t)^{-1} \\
		 \end{array}
	\right)\]
	is in the kernel of $\scrT(\C^\top)$. Therefore the map $\mu\mapsto \mu'$ is a bijection on the kernels of $\scrT(\C)$ and $\scrT(\C^\top)$.
	
	\vspace*{2mm}
	\noindent 	(b). By Lemma \ref{lm:kernel_N1_N2}, we know that an arbitrary element $\mu$ in $K\cap \GL({m+n},\K)$ can be written as
		\[\left(
		  \begin{array}{cc}
		   N_1 & O_{m,n}    \\
		   O_{n,m} & N_2 \\
		  \end{array}
		\right),\]
		where $N_1\in \GL(m,\K)$ and $N_2\in \GL(n,\K)$.
	
	By definition, we again have
	\[N_1^{-1} M N_2=M.\]
	Hence
	\[\Tr(M((N_1^t)^{-1} N N_2^t)^t)=\Tr(MN_2 N^t N_1^{-1})=\Tr(N_1^{-1}M N_2 N^t)=\Tr(M N^t)=0,\]
	for each $M\in \C$ and $N\in \C^\perp$. Therefore, the map $$\tilde{\mu} \colon ({\bf x},{\bf y}) \mapsto ({\bf x}(N_1^{-1})^t,{\bf y}N_2^t) $$ is a bijective $\K$-linear transformation on $\K^{m+n}$ which stabilizes the translation structure $\scrT(\C^\perp)$, because the above calculation shows that if $N\in \C^{\perp}$ then we have $S(N)^{\tilde{\mu}} = S(N_1^t(N(N_2^{-1})^t)$ and $N_1^tN(N_2^{-1})^t) \in \C^{\perp}$. Finally, it is immediate to verify that the map $\mu \mapsto \tilde{\mu}$ is an injective homomorphism from $K\cap GL(m+n,\K)$ into the stabilizer of $\scrT(\C^\perp)$ in $GL(m+n,\K)$.
\end{proof}

\begin{theorem}\label{th:kernel_from_covering}
	Assume that a rank metric code $\C$ contains the zero matrix and	
	\begin{equation}\label{eq:covering_all_vectors}
		\{\bx M : M\in \C\}=\K^n
	\end{equation}
	for each nonzero $\bx \in \K^m$. 
	Then the kernel $K$ of $\scrT(\C)$ is a skewfield and each element of $K$ can be expressed in the form
	\[\left(
		  \begin{array}{cc}
		    N_1 &     \\
		        & N_2 \\
		  \end{array}
		\right),\]
	with $N_1\in \Aut((\K^m,+))$ and $N_2\in \Aut((\K^n,+))$. In particular, if $\K$ is finite, then the kernel $K$ is a finite field containing $\K$, $N_1\in \GL(m,\K)$ and $N_2\in \GL(n,\K)$.
\end{theorem}
\begin{proof}
	Let $\mu$ be an arbitrary element of $K$. By Lemma \ref{lm:kernel_N1_N2}, $\mu$ can be written in the form
	\[\left(
	  \begin{array}{cc}
	    N_1 & {O}_{n,m} \\
	    {O}_{m,n} & N_2 \\
	  \end{array}
	\right),\]
	where $N_1\in \End((\K^m,+))$ and $N_2\in \End((\K^n,+))$.
	
	\vspace*{2mm}
	\noindent\textbf{Claim:} Suppose that $\mu$ does not map all elements in $\K^{m + n}$ to the zero vector. Then $N_2$ is not the zero map. 
	
	By way of contradiction, we assume that $N_2={O}_{n,n}$. Then we get
	\begin{equation}\label{eq:S(M)mu}
		S(M)^\mu = \{(\bx N_1, {\bf 0}) : \bx\in \K^m \}\subseteq  S(M), 
	\end{equation}
	for all $M\in \C$. It implies that $\by M = {\bf 0}$ for each $\by \in \{\bx N_1: \bx\in \K^m \}$ and any $M\in \C$. As $N_1\neq {O}_{m,m}$, there exists a nonzero vector ${\bf z}\in \{\bx N_1: \bx\in \K^m \}$. Thus $\{{\bf z}M : M\in \C \}=\{\bf 0\}$. It contradicts \eqref{eq:covering_all_vectors}.
	
	Next we proceed to show that both $N_1$ and $N_2$ are bijection. By way of contradiction, let us assume that $N_1$ is not invertible. There exists a nonzero vector $\bx \in \K^m$ such that $\bx N_1={\bf 0}$. Thus, for any $M\in \C$, 
	\[(\bx  N_1, (\bx M)  N_2)=  (\bf{0,0}),\]
	because of $S(M)^\mu\subseteq S(M)$. By \eqref{eq:covering_all_vectors}, we see that $N_2$ must be a zero map which contradicts the proved claim.
	
	Now we know that $N_1$ is invertible. Hence, for any nonzero vector $\bx \in \K^m$ and, the vector $\by:= \bx N_1$ is also nonzero. Again from $S(M)^\mu \subseteq S(M)$, we get
	\[ (\bx N_1, \bx M N_2) = (\by, \bx M N_2) = (\by, \by M).\]
	By the above equation, we see that the set $\{\bx M N_2: M\in \C \}$ and $\{\by M: M\in \C \}$ must be the same. By \eqref{eq:covering_all_vectors}, we further obtain that
	\[ \K^n = \{\bx M N_2: M\in \C \} = \{{\bf z} N_2: {\bf z} \in \K^n \}.\]
	That means $N_2$ is also invertible.
	
	To summarize, we have proved that $\mu\in K \setminus\{0\}$ is always invertible and clearly the inverse of an element in $K$ also belongs to $K$. Together with the fact that $K$ is a ring, we have shown that $K$ is a skewfield.
	
	When $\K$ is finite, it is clear that $K$ is also finite. Hence $K$ is a finite field. By Lemma \ref{lm:K_in_Kernel}, the set of matrices $\{aI_{m+n} : a\in \K \}$ forms a subfield of $K$ and $\mu$ is now also a $\K$-homomorphism of the vector space $\K^{m+n}$. Therefore $N_1$ and $N_2$ are both nondegenerate matrices over $\K$.
\end{proof}

In fact, when \eqref{eq:covering_all_vectors} does not hold, there exist rank metric codes $\C \subseteq \K^{m\times n}$ such that the kernel $K$ of $\scrT(\C)$ is not a skewfield.
\begin{example}\label{ex:kernel_not_field}
{\rm Let $\C$ be a set of matrices, each of which satisfies that the entries in its last row and last column are all $0$. It is straightforward to verify that $\scrT(\C)$ does not satisfy the covering property and its kernel $K$ contains the matrices
	\[ L_{a,b}=\left(
	  \begin{array}{cccc}
	    a & \cdots &  0 & 0 \\
	    \vdots  & \ddots & \vdots & 0 \\
	    0 & \cdots & a & 0 \\
	    0 & 0 & 0 & b \\
	  \end{array}
	\right) \]
	for $a,b\in \K$. As $L_{a,0}\cdot L_{0,b}$ equals the zero matrix, its kernel $K$ cannot be a skewfield.}
\end{example}
 
By Proposition \ref{prop:kernel_adjoint_dual} (a) and Theorem \ref{th:kernel_from_covering}, we can directly get the following result.
\begin{corollary}\label{coro:C1_and_hatC_1}
	Let $\C$ be a rank metric code in $\K^{m\times n}$. Assume that $\C$ contains the zero matrix and \eqref{eq:covering_all_vectors} holds for $\C$. Then there is a bijection between the kernels of $\scrT(\C)$ and $\scrT(\C^{\top})$.
\end{corollary}

\begin{corollary}\label{coro:strong_equivalence_kernel}
	Let $\C_1$ and $\C_2$ be in $\F_q^{m\times n}$. Assume that both $\C_1$ and $\C_2$ contain the zero matrix and \eqref{eq:covering_all_vectors} holds for $\C_1$ and $\C_2$. Suppose that $\C_1$ is strongly equivalent to $\C_2$. Then their kernels are isomorphic.
\end{corollary}
\begin{proof}
	If $\C_1$ is equivalent to $\C_2$, then the result follows directly from Lemma \ref{lm:equivalence_Codes_kernel}; if $\C_1$ is equivalent to $\C^\top_2$, then its kernel $K_{\C_1}$ is isomorphic to the kernel $K_{{\C}^{\top}_2}$ of $\scrT({\C}^{\top}_2)$. Together with Corollary \ref{coro:C1_and_hatC_1}, we see that the kernels of $\scrT(\C_1)$ and $\scrT(\C_2)$ are of the same size.
\end{proof}

\section{Nuclei of a rank metric code}\label{se:nuclei}

Let $\C \subseteq \K^{m\times n}$ be a $\K$-linear rank metric code. We define the {\it middle nucleus} of $\C$ as the following set of matrices of order $m$: 
$$N_m(\C) = \{Z \in \K^{m\times m} \, \colon \, ZC \in \C \text{ for all } C \in \C \}.$$
In the same way we say that the {\it right nucleus} of $\C$ is the following set: $$N_r(\C)= \{Y \in \K^{n\times n} \, : \, CY \in \C  \text{ for all }C \in \C \}.$$
In particular, when $\C$ defines a finite semifield $\mathbb{S}$, $N_m(\C)$ (resp.\ $N_r(\C)$) is exactly the middle (resp.\ right) nucleus of $\mathbb{S}$. 
In \cite{liebhold_automorphism_2016}, the middle nucleus (resp.\ right nucleus) is called \emph{left} (resp.\ \emph{right}) \emph{idealiser} of $\C$.

It is straightforward to note that inveritble elements in these sets define two subgroups of the automorphism group of the translation structure $\scrT(\C)$ fixing $S(O)$ and $S(\infty)$, respectively.

The middle and right nuclei of semifields are invariants under isotopism, which is the most widely investigated equivalence on semifields. They also play very important roles in distinguishing and the classification of semifields. Hence, it is natural to consider their properties for general rank metric codes. 
\begin{proposition}\label{prop:nuclei_invariants}
	For two equivalent linear rank metric codes $\C_1$ and $\C_2$ in $\K^{m\times n}$, their right (resp.\ middle) nuclei are also equivalent.
\end{proposition}
\begin{proof}
	Suppose that ${\C}_1$ and ${\C}_2$ are equivalent. By definition this means that there exists $\gamma\in \Aut(\K)$, $A\in\GL(m,\K)$ and $B\in\GL(n,\K)$ such that $$\C_2 = \{AM^{\gamma}B \, \colon \, M \in \C_1\}.$$

	An element $Z \in \K^{m\times m}$, belongs to the middle nucleus $N_m(\C_1)$ if and only if $AZ^{\gamma}A^{-1}$ belongs to $N_m(\C_2)$; this means that $N_m(\C_1)$ and $N_m(\C_2)$ are also equivalent. A similar argument can be used to prove that also $N_r(\C_1)$ is equivalent to $N_r(\C_2)$. This concludes the proof.
\end{proof}

Of course, we can also define middle and right nuclei for nonlinear codes. However, through the proof of Proposition \ref{prop:nuclei_invariants}, we see that the nuclei of nonlinear codes are not necessarily invariants under the isometry. If we just restrict the equivalence to the ``restricted equivalence'' $ \sim'$ in the sense that $\cC_1 \sim'\cC_2$ whenever there are $A\in \GL(m,\K)$ and $B\in\GL(n,\K)$ such that $\C_2 = \{AM^{\gamma}B : M \in \C_1\}$, then $\cC_1\sim'\cC_2$ implies that $N_r(\cC_1)$ and $N_r(\cC_2)$ are isomorphic and $N_m(\cC_1)$ and $N_m(\cC_2)$ are also isomorphic.

In the rest of this paper, we restrict ourselves to the investigation of the nuclei of linear rank metric codes.

When $\C$ is $\K$-linear, it is routine to verify that $N_m(\C)$ and $N_r(\C)$ are subrings of $\K^{m \times m}$ and $\K^{n\times n},$ respectively. Moreover, they both contain the zero map and $\K$ as a subfield. Hence, the code $\C$ can be seen as a left module (resp.\ a right module) over $N_m(\C)$ (resp.\ $N_r(\C)$).

Regarding the adjoint and Delsarte dual operation we have the following results.

\begin{proposition}\label{prop:Delsartedual}
	Let $\C$ be a linear rank metric code in $\K^{m\times n}$. Let $\C^{\top}$ (resp. $\C^{\perp}$) be the adjoint (resp.\ Delsarte dual) code of $\C$. Then the following statements hold: 
	\begin{enumerate}[label=(\alph*)]
		\item	$N_m(\C^{\top}) = N_r(\C)^\top$ and $N_r(\C^{\top})=N_m(\C)^\top$;
		\item $N_m(\C^{\perp}) = N_m(\C)^\top$ and $N_r(\C^{\perp})= N_r(\C)^\top$.
	\end{enumerate}
\end{proposition}
\begin{proof}
	By definition, (a) can be readily verified.

	For (b), we first observe that if $Z \in N_m(\C)$ then $Z^t$ belongs to $N_m(\C^{\perp})$; indeed, let $N\in \C^{\perp}$, i.e., $\Tr(CN^t)=0$ for all $C \in \C$. We have $$\Tr(C(Z^tN)^t)=\Tr(C(N^tZ))=\Tr((CN^t)Z)=\Tr(Z(CN^t))=\Tr((ZC)N^t)=0$$ for each $C \in \C$. Since the Delsarte dual operation is involutory, we have that $N_m(\C)^\top = N_m(\C^{\perp})$.
	
	It is not difficult to see that the adjoint operation and the Delsarte duality commute, i.e., $\C^{\perp \top}=\C^{\top \perp}$.  With this in mind we have the following 
	$$N_r(\C^{\perp})^\top=N_m(\C^{\perp \top})=N_m(\C^{\top \perp})=N_m(\C^{\top})^\top= N_r(\C).$$
	This concludes the proof.
\end{proof}

As in the previous section on kernels, we are curious about the conditions under which middle or right nucleus of a code is a field.

\begin{lemma}\label{lm:middle_nuclei_field}
	Let $\C$ be a linear rank metric code of $\K^{m \times n}$ with $m\leqslant n$ and its minimum distance  $d \geqslant \left\lfloor \frac{m}{2} \right\rfloor+1$. Assume that there is at least one full rank matrix in $\C$. For any element $Z \in N_m(\C)$, assume that there exists a nonzero $C_0 \in \C$ such that $ZC_0=O$. Then $Z$ is the zero matrix $O$. In particular, when $\C$ is a finite set, all nonzero matrices in $N_m(\C)$ are invertible and $N_m(\C)$ is a field.
\end{lemma}
\begin{proof}
	By $ZC_0=O$, the matrix $Z \in \K^{m \times m}$ can not have full rank. That means $d'<m$, where $d':= \rk(Z)$. 
	
	By way of contradiction, we assume that $Z\neq O$. As a full rank matrix $M$ is assumed to be in $\C$,  we have $ZM\neq O$. Since $ZM\in \C$,  $\rk(ZM)\geqslant d$ and $d'\geqslant d$.
	
	Again from $ZC_0=O$ we also have that $\rk(C_0)\leqslant m-d'$. Together with $d'\geqslant d$ we have $$d\leqslant \rk(C_0)\leqslant m-d'\leqslant m-d.$$ This contradicts the assumption that $d \geqslant \left\lfloor \frac{m}{2} \right\rfloor+1$.
	
	Now we suppose that $\C$ is finite. If $Z$ is degenerate, then $ZM'$ is not full rank for every $M'\in \C$, which implies that $Z\C \subsetneq \C$. Since $\C$ is finite and linear, there exists a nonzero matrix $C_0$ such that $ZC_0=O$. From the previous part, we know that $Z$ must be zero. Hence the nonzero matrices in $N_m(\C)$ are all nondegenerate. As $N_m(\C)$ is finite, closed under addition and multiplication and it contains the identity matrix, $N_m(\C)$ is a field.
\end{proof}

By transposition we get:
\begin{lemma}\label{lm:right_nuclei_field}
	Let $\C$ be a linear rank metric code of $\K^{m \times n}$ with $\left\lfloor\frac{n}{2} \right\rfloor+1\leqslant m\leqslant n$ and its minimum distance  $d \geqslant \left\lfloor \frac{n}{2} \right\rfloor+1$. Assume that there is at least one full rank matrix in $\C$. For any element $Z \in N_r(\C)$, assume that there exists $C_0 \in \C\setminus \{{O}\}$ such that $C_0Z={O}$. Then $Z$ is the zero matrix. In particular, when $\C$ is a finite set, all nonzero matrices in $N_r(\C)$ are invertible and $N_r(\C)$ is a field.
\end{lemma}

\section{Kernels and nuclei of MRD codes}\label{se:MRD}
In this section, we investigate the kernel and nuclei of an MRD code over a finite field. 

\subsection{Kernels of MRD codes}First let us consider the kernel of an MRD code.

\begin{theorem}\label{th:MRD_covering}
	Let $\C$ be an MRD code in $\F_q^{m\times n}$. Then $\scrT(\C)$ satisfies the covering property, i.e., for any nonzero vector $\bx\in\F_q^m$ and any $\by \in \F_q^n$, there is at least one matrix $M\in \C$ such that $\bx M = \by$.
\end{theorem}
\begin{proof}
	Without loss of generality, we assume that $\bx=(1,0, \dots, 0)$; otherwise we choose an invertible matrix $L$ such that $\bx L=(1,0\dots, 0)$ and left multiply its inverse matrix $L^{-1}$ by $M\in \C$ to get another MRD code.
	
	Assume, by way of contradiction, that there is an element $\by\in \F_q^n$ such that $\bx M\neq \by$ for all $M\in \C$. Suppose that the minimum rank distance of $\C$ is $d$ and $m\leqslant n$. It means that there are $q^{n(m-d+1)}$ matrices in $\C$.
	
	For each ${\bf z}\in \F_q^n$, we take $U_{\bf z}:= \{ M\in \C : {\bf x} M = {\bf z} \}$. It is clear that
	\begin{equation*}
		\sum_{{\bf z}\in \F_q^n} \#U_{\bf z} = q^{n(m-d+1)},
	\end{equation*}
	and
	\[ \#U_{\bf y} = 0.\]
	From them, we can derive that
	\begin{equation}\label{eq:count_max_uz}
		\max_{{\bf z}\in \F_q^n}\{ \#U_{\bf z}\} \geqslant \frac{q^{n(m-d+1)}}{q^n-1}>q^{n(m-d)}. 
	\end{equation}
	Let $\bar{\bf z}$ be the vector such that $\#U_{\bar{\bf z}}=\max_{{\bf z}\in \F_q^n}\{ \#U_{\bf z} \}$. 
	
	Now let us look at the matrices in $U_{\bf \bar{z}}$. As $\bx = (1,0,\dots, 0)$, the first row of each $M\in U_{\bf \bar{z}}$ equals $\bf {\bar z}$. For any $m-d$ rows except for the first one,  by \eqref{eq:count_max_uz}, we see that there must exist two matrices $M$ and $M'$ in $U_{\bf \bar{z}}$ such that these $m-d$ rows are the same. It follows that the rank of $M-M'$ is at most $d-1$, which contradicts the assumption that $\C$ is an MRD code.
	
	For the $m>n$ case, we can similarly prove that there exist two matrices $M$ and $M'$ in which the first $\lfloor  \frac{m}{n}(n-d)+1 \rfloor$ rows are the same, which contradicts the minimum distance of $\C$.
\end{proof}

Theorem \ref{th:MRD_covering} can also be derived from the fact that any MRD code of $\F_{q}^{m\times n}$ with minimum distance $d$ is an \emph{$(n-d+1)$-design} of index $1$ in $\F_{q}^{m\times n}$; see \cite[Section 5]{delsarte_bilinear_1978} for more details.

\begin{corollary}
	Let $\C$ be an MRD code in $\F_q^{m\times n}$ with $O\in \C$, such that $\gcd(m,n)=1$ or the minimum distance $d<\min \{m,n\}$. Then, the kernel $K$ of $\scrT(\C)$ is $\F_q$.
\end{corollary}
\begin{proof}
	By Lemma \ref{lm:K_in_Kernel}, $K$ contains a subfield isomorphic to $\F_q$. By Theorems \ref{th:kernel_from_covering} and \ref{th:MRD_covering}, we know that $K$ is a finite field. Let us say $K=\F_{q^r}$ for a positive integer $r$. By the definition of the kernel the set $S(\C)$ can be viewed as a set of $K$-subspaces in $\K^m \times \K^n$. That means each matrix $M$ in $\C$ can also be viewed as a matrix over $\F_{q^r}$. It implies that $r$ divides $m$ and $n$.
	
	When $\gcd(m,n)=1$, it is clear that $r$ must be $1$. When the minimum distance $d<\min \{m,n\}$, $r=1$ can be derived from the fact that there exist matrices of rank $\min\{m,n\}$ and $\min\{m,n\}-1$ in $\C$  by Lemma \ref{lm:MRD_weight_all}.
\end{proof}

It is worth pointing out that when $\min\{m,n\}=d$, the kernel of $\scrT(\C)$ can be strictly larger than $\F_q$. For instance, when $m=n=d$, an MRD code $\C$ is exactly a semifield, and the kernel of $\scrT(\C)$ corresponds to the so-called left nucleus of the semifield.  There always exist semifields of order $q^n$ with left nucleus larger than $q$; for instance the famous Albert's twisted fields \cite{albert_generalized_1961,biliotti_collineation_1999}.

When $\C$ is not an MRD code, there are also examples whose kernels are strictly larger than $\F_q$.
\begin{example}\label{ex:kernel_larger}
	{\rm Let $n=4$. Let $\C$ be a set of $4\times 4$ matrices over $\F_q$ derived form the following set of linearized polynomials in $\F_{q^4}[X]$:
	\[ \{ a_0 X + a_1 X^{q^2} : a_0,a_1\in \F_{q^4} \}. \]
	Let $c$ be an element of $\F_{q^2}$. For any $a_0,a_1,x\in \F_{q^4}$, we always have
	\[a_0 (cx) + a_1 (cx)^{q^2}=c(a_0 x + a_1 x^{q^2}).\]
	It implies that $\F_{q^2}$ is a subfield of the kernel of $\C$.}
\end{example}

\subsection{Nuclei of MRD codes}
For the nuclei of MRD codes, we can prove the following results:

\begin{theorem}\label{th:nuclei_of_MRD}
	Let $\C$ be a linear MRD code in $\F_q^{m\times n}$ with $m\leqslant n$ and  minimum distance $d>1$. Then the following statements hold:
	\begin{enumerate}[label=(\alph*)]
		\item Its middle nucleus $N_m(\C)$ is a finite field.
		\item When $\max\{d,m-d+2\} \geqslant \left\lfloor \frac{n}{2} \right\rfloor +1$, its right nucleus $N_r(\C)$ is a finite field.
	\end{enumerate}
\end{theorem}
\begin{proof}
	(a) When $d\geqslant  \left\lfloor \frac{m}{2} \right\rfloor +1$, it is already proved in Lemma \ref{lm:middle_nuclei_field}, because $\C$ is a finite set and there is at least one  full rank matrix in $\C$ by Lemma \ref{lm:MRD_weight_all}; when $d< \left\lfloor \frac{m}{2} \right\rfloor +1$, we look at its Delsarte dual $\C^\perp$. By \eqref{eq:delsarte_dual_distance}, its distance 
	$$d(\C^\perp) =m-d+2> m- \left\lfloor \frac{m}{2} \right\rfloor+1 \geqslant \left\lfloor \frac{m}{2} \right\rfloor+1 .$$
	Again by Lemma \ref{lm:middle_nuclei_field}, we have $N_m(\C^\perp) $ is a finite field. As  $N_m(\C^\perp)=N_m(\C)^\top$ (Proposition \ref{prop:nuclei_invariants} (b)), $N_m(\C)$ is also a finite field.
	
	(b) When $d\geqslant  \left\lfloor \frac{n}{2} \right\rfloor +1$, we get it by Lemma \ref{lm:right_nuclei_field}; otherwise $m-d+2 \geqslant\left\lfloor \frac{n}{2} \right\rfloor +1$, we have that $N_r(\C^\perp)$ is a finite field. From $N_r(\C^\perp)=N_r(\C)^\top$ (Proposition \ref{prop:nuclei_invariants} (b)), we see that $N_r(\C)$ is also a finite field.
\end{proof}

\begin{remark}
\begin{enumerate}[label=(\alph*)]
	\item	When the minimum distance of an MRD code $\C$ is $d=1$, $\C$ is the whole space $\K^{m\times n}$. Then $N_m(\C)=\K^{m\times m}$ and $N_r(\C)=\K^{n\times n}$.
	\item 	When the conditions in Theorem \ref{th:nuclei_of_MRD} are satisfied for a linear MRD code $\C$, it can be viewed as a left vector space over $N_m(\C)$ as well as a right vector space over $N_r(\C)$.
\end{enumerate}
\end{remark}

When $m=n$, it is easy to get the following result from Theorem \ref{th:nuclei_of_MRD}.
\begin{corollary}
	Let $\C$ be a linear MRD code in $\F_q^{n\times n}$ and let the minimum distance $d >1$. Then its middle nucleus and right nucleus are both finite fields.
\end{corollary}

In general, Theorem \ref{th:nuclei_of_MRD} (b) does not hold when $\max\{d,m-d+2\} < \left\lfloor \frac{n}{2} \right\rfloor +1$. Let us look at an example with $m=2$, $n=4$, $q=2$ and $d=2$.
\begin{example}
{\rm Let $\F_{q^2}\cong {\mathcal K} \subseteq \F_q^{2 \times 2}$ (for instance, ${\mathcal K}=\F_q[T]$, where $T$ is an irreducible operator over $\F_q$).

A rank metric code $\C\subseteq \F_{2}^{2\times 4}$ is defined as
	\[\C:=\{(B_1, B_2) : B_1,B_2\in {\mathcal K} \},\]
	where $(B_1,B_2)$ stands for the $2\times 4$ matrix whose first $2\times 2$ block is $B_1$ and second $2\times 2$ block is $B_2$.
	
	Clearly all nonzero matrix in $\C$ is of full rank. As there are totally $16$ matrices in $\C$ and $q^{\max\{m,n\}(\min\{m,n\}-d+1)}=16$, $\C$ is an MRD code.
	
	Let 
	$$Z= \left(
	\begin{array}{cccc}
	0 & 0 & 0 & 0 \\ 
	0 & 0 & 0 & 0 \\ 
	0 & 0 & 1 & 0 \\ 
	0 & 0 & 0 & 1
	\end{array} 
	\right).$$
	Then $CZ \subseteq \C$ for each $C\in \C$. However $\rk(Z)=2$.}
\end{example}

\subsection{Nuclei of known linear MRD codes}\label{subsec:nuclei_MRD}
Observe that when $n=m$, it does not matter which linearly independent elements $\alpha_1, \dots, \alpha_n$ are chosen in \eqref{eq:mn_MRD}, because the derived codes are equivalent by multiplying a certain invertible matrix. Thus a generalized Gabidulin code can be directly described as the set of polynomials in (\ref{eq:GG}).  Now let us first restrict ourselves to MRD codes defined through sets of linearized polynomials.

Besides $\cG_{k,s}$ defined by \eqref{eq:GG}, there are two other sets of linearized polynomials which define MRD codes for arbitrary values of $n$ and $k$. These were recently obtained in \cite{sheekey_new_2015}. Precisely, let $n,k,h\in \Z^+$ and $k<n$. Let $\eta$ be in $\F_{q^n}$ such that $N_{q^n/q}(\eta)\neq (-1)^{nk}$. Then the set
\begin{equation}\label{eq:Sheekey}
	\cH_k(\eta, h) = \{a_0 x + a_1 x^q + \cdots +a_{k-1} x^{q^{k-1}} + \eta a_0^{q^h} x^{q^k}: a_0,a_1,\dots, a_{k-1}\in \F_{q^n} \}
\end{equation}
is an $\F_q$-linear MRD code of size $q^{nk}$; these are called \emph{twisted Gabidulin} codes.

Also in \cite{sheekey_new_2015} the following generalization of these examples was mentioned. Let $n,k,s,h\in \Z^+$ satisfying that $\gcd(s,n)=1$ and  let  $\eta$ be in $\F_{q^n}$ such that $N_{q^{sn}/q^s}(\eta)\neq (-1)^{nk}$. Then the set 
\[\cH_{k,s}(\eta, h) = \{a_0 x + a_1 x^{q^s} + \dots +a_{k-1} x^{q^{s(k-1)}} + \eta a_0^{q^h} x^{q^{sk}}: a_0,a_1,\dots, a_{k-1}\in \F_{q^n} \}\]
is an $\F_q$-linear MRD code of size $q^{nk}$.  These sets $\cH_{k,s}(\eta, h)$ latter are known as \emph{generalized twisted Gabidulin} codes after \cite{lunardon_generalized_2015}, where they were intensively studied. Precisely, in \cite{lunardon_generalized_2015} the automorphism group of a generalized twisted Gabidulin code was completely determined and it was proven that the relevant family contains the two known classes $\cG_{k,s}$ and $\cH_k(\eta, h)$ of MRD codes as proper subsets.

Let $\cC$ and $\cC'$ be two set of $q$-polynomials over $\F_{q^n}$. It is clear that $\cC$ and $\cC'$ define two rank metric codes in $\F_q^{n\times n}$ and they are equivalent if there exist two permutation $q$-polynomials $L_1$, $L_2$ and $\rho\in \Aut(\F_q)$ such that $\cC' =\{ L_1\circ f^\rho \circ L_2(x) : f\in \cC \}$, where $(\sum a_{i}x^{q^i})^\rho:= \sum a_{i}^\rho x^{q^i}$ and the symbol $L\circ L'$ for two $q$-polynomials $L$ and $L'$ denotes the polynomial $L(L'(x))$. In particular, the automorphism group of the code derived from $\cC$ consists of all $(L_1,L_2,\rho)$ fixing $\cC$. From the proof of Theorem 4.4 in \cite{lunardon_generalized_2015}, the automorphism group of $\cH_{k,s}(\eta, h)$ can be completely determined. 
\begin{theorem}\label{th:automorphism_Hks}
	Let $n,k,s,h\in \Z^+$ satisfying $\gcd(n,s)=1$ and $2\le k\le n-2$. Let $\eta$ be in $\F_{q^n}$ satisfying $N_{q^{sn}/q^s}(\eta)\neq (-1)^{nk}$. Then $(L_1,L_2,\rho)$ is an automorphism of $\cH_{k,s}(\eta, h)$ if and only if there exist $c,d\in \F_{q^n}^*$ and $r\in \{0,1,\dots, n-1\}$ such that $L_1=cx^{q^r}$, $L_2=dx^{q^{n-r}}$ and
	\begin{equation}\label{eq:automorphism}
		\eta c^{q^h-1}d^{q^{r+h}-q^{r+sk}} = \eta^{\rho q^r}.
	\end{equation}
\end{theorem}

In what follows we will determine the middle nucleus and the right one of $\cH_{k,s}(\eta,h)$. To this aim, it makes sense first to describe the nuclei in the context of $q$-polynomials over $\F_{q^n}$. 

Regard to this, denote by $\cC \subseteq \mathbb E=\End_{\F_q}(\F_{q^n})$ the set of $q$-polynomials defining a code $\C \subseteq \F_q^{n\times n}$.  Clearly, we have that $N_m(\C)  \cong {\mathcal N}_m(\cC)  = \{\varphi \in \mathbb E \, \colon \,  f \circ \varphi \in \cC \text{ for all } f \in \cC \}$ and $N_r(\C) \cong \mathcal N_r(\cC)  = \{\varphi \in \mathbb E \, \colon \,  \varphi \circ f \in \cC \text{ for all } f \in \cC \},$ where the symbol $\circ$ stands for the composition of maps. By definition and Theorem \ref{th:nuclei_of_MRD}, for each $f\in {\mathcal N}_m(\cC)$ and each $g\in {\mathcal N}_r(\cC)$, $(x, f, \id)$ and $(g, x, \id)$ are both automorphisms of $\cC$.

By Theorem \ref{th:automorphism_Hks}, we can get the following results:
\begin{corollary}\label{coro:nucleus_GTGC}
Let ${\mathcal H}_{k,s}(\eta, h)$ be a generalized twisted Gabidulin code. Then we have
\begin{enumerate}[label=(\alph*)]
\item if $\eta =0$, then ${\mathcal H}_{k,s}(0, h)={\cG}_{k,s}$ and ${\mathcal N}_m({\cG}_{k,s}) = \mathcal N_r({\cG}_{k,s}) \cong \F_{q^{n}}$;
\item if $\eta \neq 0$, then ${\mathcal N}_m({\mathcal H}_{k,s}(\eta, h))\cong \F_q^{\gcd(n,sk-h)}$ and $\mathcal N_r({\mathcal H}_{k,s}(\eta, h))\cong \F_q^{\gcd(n,h)}$.
\end{enumerate}
\end{corollary}
\begin{proof}
	To determine the middle nucleus, we only have to check the automorphisms of the form $(x, f, \id)$. Let $\rho$ to be the identity map, $L_1=x$ and $L_2=dx$. If $\eta=0$, then \eqref{eq:automorphism} is always satisfied; otherwise, \eqref{eq:automorphism} becomes
	\[\eta d^{q^{h}-q^{sk}} = \eta,\]
	which holds if and only if $d\in \F_q^{\gcd(n,sk-h)}$.
	
	To determine the right nucleus, we let $\rho$ to be the identity map, $L_2=x$ and $L_1= cx$. Now if $\eta=0$, then \eqref{eq:automorphism} is always satisfied; otherwise, we have
	\[\eta c^{q^h-1} = \eta,\]
	which holds if and only if $c\in \F_q^{\gcd(n,h)}$.
\end{proof}

Now let us turn to linear MRD codes in $\F_q^{m\times n}$ with $m<n$. Most of MRD codes with $1<k<n-1$ and $m<n$ are in the following form:
\begin{equation}\label{eq:mn_MRD_cH}
\left\{ \left(\bv(f(\alpha_1)), \dots, \bv(f(\alpha_m))\right)^t: f\in \cH_{k,s}(\eta, h)
  \right\},
\end{equation}
where $\alpha_1,\dots,\alpha_m$ are linear independent. Several new constructions of MRD codes which are not in this form are presented recently in \cite{horlemann-trautmann_new_2015} and they are proved to be not equivalent to any Gabidulin code. However, we do not know whether they are equivalent to a generalized twisted Gabidulin code \eqref{eq:mn_MRD_cH} or not.

Let $\xi$ be a primitive element of $\F^*_{q^{n}}$ and
\[\mathbb{H} := \left\{ \left(\bv(f(1)),\bv(f(\xi)), \dots, \bv(f(\xi^{n-1}))\right)^t: f\in \cH_{k,s}(\eta, h)
  \right\},\]
then by multiplying a suitable $m$ by $n$ matrix $L$ of rank $m$ on the left of elements in $\mathbb H$, we can get \eqref{eq:mn_MRD_cH}.  In another word, the MRD code \eqref{eq:mn_MRD_cH} is the image of $\mathbb{H}$ under a projection from $\F_q^{n\times n}$ to $\F_q^{m\times n}$.

In \eqref{eq:mn_MRD}, if $\eta=0$, i.e.\ $\cH_{k,s}(\eta, h)= \cG_{k,s}$, its middle and right nuclei are determined very recently in \cite{liebhold_automorphism_2016}; see \cite{morrison_equivalence_2014} for the calculation of the middle nuclei too. Notice that, in \cite{liebhold_automorphism_2016}, the (generalized) Gabidulin code is described as the adjoint of \eqref{eq:mn_MRD}. Hence the right (resp.\ left) idealiser there is exactly the middle (resp.\ right) nucleus of \eqref{eq:mn_MRD}. 
By  Corollary \ref{coro:nucleus_GTGC} and the following lemma which can be directly obtained by definition, we can also easily show that the right nucleus of \eqref{eq:mn_MRD} always contains $\F_{q^n}$.

\begin{lemma}\label{lm:projection_transposed_nucleus}
	Let $\C$ be a rank metric code in $\K^{m\times n}$. Let $L$ be an $\ell \times m$ matrix with $\ell <m$. Then 
	\[ N_r(\C) \subseteq N_r(\{LC : C\in \C \}).  \]
\end{lemma}

For the middle nucleus of a projection of a given code, it seems difficult to get any general result similar to Lemma \ref{lm:projection_transposed_nucleus}. After a projection, the new middle nucleus is in the set of matrices of a smaller size. However, it is not necessary that the cardinality of the middle nucleus is getting smaller. For instance, the map from $\F_{p^n}^2$ to itself given by
\[  
(x,y)\mapsto \left((a^{p^k} x + x^{p^k} a) + \alpha (b^{p^k} y + y^{p^k} b)^\sigma, ay + bx\right),
\]
for any $a,b\in \F_{p^n}$, where $2\nmid p$, $2\nmid \frac{n}{\gcd(n,k)}$, $\sigma\in \Aut(\F_{p^n})$ and $\alpha$ is a nonsquare in $\F_{p^n}$, comes from the commutative semifields constructed in \cite{zhou_new_2013}. The middle nucleus of this semifield, which is exactly the middle nucleus of the derived MRD code, is $\F_{p^{\gcd(n,k)}}$ if $\sigma$ is nontrivial or $\F_{p^{2\gcd(n,k)}}$ if $\sigma$ is trivial. If we project it to the last $n$ rows, then we only have
the matrices corresponding to 
\[(x,y)\mapsto (ay+bx).\]
It is easy to show that the middle nucleus of this new set of matrices is $\F_{p^n}$. Hence, if $2\gcd(n,k)<n$, the new middle nucleus is larger than the original one. 

By looking at the projection of rank metric codes, we may also find some small structures just as we have shown for some semifields. The idea of projection and lifting have been already applied several times in the constructions of APN functions and semifields; see \cite{browning_apn_2010,budaghyan_constructing_2009,edel_new_2009,hou_switchings_2015,pott_switching_2010}.

As the middle nuclei and the right ones are both invariant with respect to the equivalence on rank metric codes, we may also consider the set of the middle (resp.\ right) nuclei of every projection of a rank metric code. More precisely, let $\C$ be a rank metric code in $\K^{m\times n}$. For any $l<m$ and any $l$-dimensional subspace $U$ of $\K^m$, we choose a matrix $L_U\in \K^{l\times m}$ whose rows form a basis of $U$. It is not difficult to see that for a given subspace $U$, distinct ways of choosing $L_U$ do not affect $N_m(L_U \C)$ and $N_r(L_U \C)$ up to equivalence. The \emph{middle nuclei spectrum} of a linear rank metric code $\C\subseteq \K^{m\times n}$ is the multiset defined by
\[~\left\{*~~ ( l ,N_m(L_U \C)) : 1<l<m, U\text{ is an  $l$-dimensional subspace of } \K^m    ~~*\right\}.   \]
Similarly, we can define the \emph{right nuclei spectrum} of $\C$. It is clear that these two spectra are both invariants with respect to the equivalence on rank metric codes. Hence they are useful for telling whether two codes are equivalent or not.

It is in general also not easy to compute these spectra for a linear rank metric code. We can use computer to get them for some MRD codes with small parameters.
\begin{example}
	{\rm Let $q=3$, $m=n=4$, $k=2$ and $s=h=1$. 	Let $\eta$ be a root of $X^4-X^3-1\in \F_3[X]$. Then $\cH_{k,s}(\eta, h)$ defines an MRD code $\C$ in $\F_3^{4\times 4}$. 
	
	For $l=3$, there are totally $40$ subspaces $U$ of dimension $l$ in $\F_3^4$. For each of such subspace $U$, our MAGMA program shows that $N_m(L_U \C)\cong \F_3$ and $N_r(L_U \C)\cong \F_3$. When $l=2$ and $1$, for each subspace $U$ of dimension $l$, we have $L_U\C = \F_3^{l\times 4}$ from which it follows  $N_m(L_U \C)=\F_3^{l\times l}$ and $N_r(L_U \C)=\F_3^{4\times 4}$.
	
	If we take $\eta=0$, then $\cH_{k,s}(\eta, h)=\cG_{k,s}=\cG_{2,1}$. Let us use $\C'$ to denote the MRD code in $\F_{3}^{4\times 4}$ corresponding to it. For each subspace $U$ of dimension $3$, Lemma 4.1 and Theorem 4.5 in \cite{liebhold_automorphism_2016} tell us that $N_m(L_U \C')\cong \F_3$ and $N_r(L_U \C')\cong \F_{3^4}$. Again when $l=1,2$, for each subspace $U$ of dimension $l$, we have $L_U\C' = \F_3^{l\times 4}$ which means $N_m(L_U \C')=\F_3^{l\times l}$ and $N_r(L_U \C')=\F_3^{4\times 4}$.}
\end{example}

\section{dimensional dual hyperovals, their kernels and nuclei}\label{se:DHO}
Let $U$ be an $(n+r)$-dimensional vector space over $\F_q$ with $n>1$ and $r\geq 1.$ A collection $\mathbb{D}$ of $n$-dimensional subspaces of $U$ for $n\ge 2$ is called a {\it dimensional dual hyperoval of rank $n$} (abbreviated to DHO) if the following conditions are satisfied: 
\begin{enumerate}
	\item[(D1)] $\dim(X_1 \cap X_2) =1$,\, for each pair of elements $X_1$ and $X_2$ in $\mathbb{D}$;
	\item[(D2)] $X_1 \cap X_2 \cap X_3 = \{\bf 0\}$, for any mutually distinct $X_i\in \mathbb{D}$ ($i\in \{1,2,3\}$);
	\item[(D3)] $\#\mathbb{D}=q^{n-1}+q^{n-2}+\cdots+q+2$. (Observe that $\#\bD =2^n$ if $q=2$.)
\end{enumerate}

The \emph{ambient space} of $\mathbb{D}$, denoted by the symbol $\langle \mathbb{D} \rangle$, is the subspace of $U$ spanned by the elements of $\mathbb{D}$. The subspaces in $\bD$ are called the \emph{components}. Often, a DHO of rank $n$ is viewed projectively and called an $(n-1)$-dimensional dual hyperoval.  Yoshiara \cite{yoshiara_ambient_2004} shows that $n-1\le r \le n(n-1)/2$ if $q\neq 2$ and $n-1\le r \le n(n-1)/2+2$ if $q=2$ (however it is conjectured the upper bound $n(n-1)/2$ also holds when $q=2$).

Up to now, no DHO over a field of odd characteristic is discovered. For any $n\ge 2 $ and any even $2$-power $q$, DHOs of rank $n$ over $\F_q$ are known. There are various constructions of DHOs, see \cite{dempwolff_dimensional_2013,dempwolff_dimensional_2014,dempwolff_orthogonal_2015,taniguchi_new_2014,taniguchi_bilinear_2016,yoshiara_dimensional_2006,yoshiara_dimensional_2008} for instance.

By applying all the translations of the ambient space $V:=\langle \mathbb{D} \rangle$ to the subspaces in an DHO $\bD$, we obtain a translation structure $\scrT_\bD$. According to definition, its kernel $K$ is the set of all endomorphisms of the group $(V,+)$ such that $X^\mu \subseteq X$, for all $X\in \bD$. In the following we determine the kernel of $\scrT_\bD$. 

\begin{proposition}\label{prop:DHOkernel}
Let $\mathbb{D}$ be a DHO of rank $n$ and $V:=\langle \mathbb{D} \rangle$ the ambient space of $\bD$ . Then the kernel $K$ of $\scrT_\bD$ is isomorphic to $\F_q$.
\end{proposition}
\begin{proof}
Clearly, $K$ is a subring of $\End_{\F_q}(V)$ and $\{\lambda 1_V \,|\, \lambda \in \F_q\}$ is a subfield of $K$. 

On the other hand, by Conditions (D1), (D2) and (D3), it is straightforward to see that, for any $X\in \bD$, each point in $X\setminus \{\bf{0}\}$ is covered by exactly one of the $1$-dimensional subspaces in $\{X\cap Y : X,Y \in \bD \}$. Furthermore, every element $\mu \in K$ fixes each $1$-dimensional subspace $X \cap Y$, $X, Y \in \mathbb{D}$. It follows that, if we regard $X\setminus\{\bf 0\}$ as a projective space, then by the fundamental theorem of projective geometry, $\mu$ induces a scalar on $X$. By Condition (D2), $\mu$ induces the same scalar on $V$.
\end{proof}

We say that $\mathbb{D}$  \emph{splits} over the $r$-dimensional subspace $Y \subseteq U$, if $U=X \oplus Y$ for all $X \in \mathbb{D}$; all known DHOs split over some subspace of their ambient space. In such a case we can identify $U$ with the Cartesian product $\{(\bx,\by) : \bx\in X,\, \by\in Y\}$. In particular, when $q=2$, it is not difficult to verify that there exists an injective map $\beta$ from $X$ into  $\Hom(X,Y)$ such that every member of $\mathbb{D}$ can be written in the following fashion 
\[X({\bf a}):=\{(\bx, \bx\beta({\bf a})) : \bx \in X\},\] 
for some ${\bf a}\in X$. In particular, $\{(\bx,{\bf 0}): \bx \in X\}=X({\bf 0})$, as $\beta({\bf 0})$ is the zero map. 

The subset $\{\beta({\bf a}) : {\bf a}\in X\}$ of $\Hom(X,Y) \cong  \F_2^{n\times r}$, satisfies the following properties corresponding to Conditions (D1) and (D2) stated above for a DHO: 
\begin{enumerate}[label=(P\arabic*)]
\item\label{item.P1} The rank of $\beta({\bf a})-\beta({\bf b})$ is $n-1$ for distinct ${\bf a}, {\bf b} \in X$.
\item\label{item.P2} For each ${\bf a}\in X$, the map sending any ${\bf b}\in X\setminus \{{\bf a}\}$ to the kernel of $\beta({\bf a})-\beta({\bf b})$ is a bijection from  $X\setminus \{{\bf a}\}$ to the set of 1-dimensional subspaces of $X$.
\end{enumerate}

Conversely, a subset of $\Hom(X,Y)$ indexed by the elements in $X$ and satisfying Conditions \ref{item.P1} and \ref{item.P2} stated above, determines a DHO of rank $n$ over $\F_2$ which contains $X$ as a member and splits over $Y$. In some references such as \cite{dempwolff_orthogonal_2015}, such a set is called a \emph{DHO-set}.  

Hence, if $q=2$ and $\mathbb{D}$ is a DHO of rank $n$ in $U$ which splits over $Y$, its associated DHO-set is ${\mathcal D}=\{\beta({\bf a}) : {\bf a} \in X\}$. 
In view of Condition \ref{item.P1}, ${\mathcal D}$ can be seen as a rank metric code in $\Hom(X,Y)\cong \F_2^{n \times r}$ (containing the zero matrix) with minimum distance $n-1$ and $\#{\mathcal D}=2^n$. We observe that ${\mathcal D}$ is an MRD code when $r=n-1$. Also, we have that  $\scrT_\bD=\scrT(\cD)$ and as a consequence of Proposition \ref{prop:DHOkernel}, we may state the following result.
\begin{corollary}\label{coro:kernel_DHO}
	Let $\mathcal D$ be a DHO-set associated with a DHO $\bD$ of rank $n$ in $U:=\langle \bD \rangle\cong \F_2^{n+r}$. Then the kernel of $\scrT(\mathcal D)$ is $\F_2$.
\end{corollary}

\subsection{Bilinear DHOs, their kernels and nuclei}\label{Bilinear DHOs, their kernels and nuclei}
A DHO $\mathbb{D}$ is called {\it bilinear} if the map $\beta$ mentioned above, is $\F_2$-linear; or, in other words, if the subspace $Y \subseteq U$ can be chosen in such a way that the DHO-set associated to $\mathbb{D}$ turns out to be an abelian group. Bilinear DHOs only exists for $q=2$ and in such a case we have that ${\mathcal D}$ is a linear code in $\F_2^{n \times r},$ containing the zero matrix, with minimum distance $d=n-1$ and dimension $n$. Also, $X\setminus \{0\}$ is a disjoint union of $\ker(\beta({\bf a}))\setminus \{0\}$ with ${\bf a} \in X\setminus \{0\}$. 

Before going on with linear codes associated with DHO sets, let us consider again the slightly more general situation. 

Suppose that the rank metric code $\C \subset  \F_q^{n+r}$ is a $\F_q$-subspace of dimension $\ell$ of $\F_q^{n\times r}$. Also in this case there exists a $\F_q$-linear injection $\beta: \F_q^{\ell} \rightarrow \F_q^{n\times r}$ such that $\C=\{ \beta(\ba): \ba \in \F_q^{\ell} \}.$ In this way one can set up a bilinear map $\sigma(\cdot,\cdot) \, \colon \, \F_q^n \times \F_q^{\ell} \rightarrow \F_q$, via the rule $\sigma({\bf x}, {\bf y}) = {\bf x}\beta({\bf y})$. 

From $\beta$ we may define a new map $\beta^{\circ} : \F_q^n \rightarrow \F_q^{\ell \times r}$ (called the \emph{opposite} to $\beta$) and hence a new bilinear function by the following rule 
$$\bx\beta^{\circ}({\bf y})={\bf y}\beta(\bx) \,\,\, \text{for} \,\, \bx \in \F_q^{\ell}, {\bf y} \in \F_q^n.$$ 
We put $\mathcal {\C}^{\circ} =\{\beta^{\circ}({\bf x}) :{\bf x} \in \F_q^n\}$, and refer to it as to the {\it opposite} code to $\C$. 

On the other hand, we have another map $\beta^{\dagger}$ from $\F_q^{\ell}$ to $\F_q^{r \times n}$. Precisely, fix two non-degenerate symmetric bilinear forms $\bili_{\F_q^n}(\cdot,\cdot)$ on $\F_q^n$ and $\bili_{\F_q^r}(\cdot,\cdot)$ on $\F_q^r$, and for each ${\bf a} \in \F_q^{\ell}$ denote by $\beta^{\dagger}({\bf a})$ the \emph{adjoint operation} of $\beta({\bf a})$ with respect to $\bili_{\F_q^n}$ and $\bili_{\F_q^r}$; i.e. the element in $\F_q^{r \times n}$ satisfying the equation 
$$\bili_{\F_q^r}(\bx\beta^{\dagger}({\bf a}),\by)+\bili_{\F_q^n}(\bx,\by\beta({\bf a}))={\bf 0} \,\,\, \text{for} \,\, \bx \in \F_q^n ,{\bf a} \in \F_q^{\ell}, \, \by\in \F_q^r.$$
We set ${\mathcal C}^{\dagger}:=\{\beta^{\dagger}(\ba) : \ba \in \F_q^m\}$.

Appropriate $\F_q$-bases of $\F_q^n$ and $\F_q^r$ can be chosen in such a way that ${\mathcal C}^{\dagger} = {\mathcal C}^{\top}$, which is exactly the adjoint code of $\mathcal C$. When $r=n$ and $\C$ defines a finite semifield $\bbS$, then  ${\mathcal C}^{\circ}$  and ${\mathcal C}^{\top}$ correspond to the spread sets associated with the semifields obtained from $\bbS$ applying so called Knuth operations introduced in \cite{knuth_finite_1965}. Also in \cite{knuth_finite_1965}, Knuth noted that there are in total five semifields which can be derived from $\bbS$ using $\circ$ and $\top$, and there is a group $G$ isomorphic to $S_3$ acting on these six semifields.

In a similar way, starting with an $\F_q$-linear rank metric code $\cD:=\{\beta(\ba): \ba\in X \}\subseteq \F_q^{n\times r}$, where $X$ and $Y$ are $n$-dimensional and $r$-dimensional over $\F_q$ respectively, and $\beta$ is an injective $\F_q$-linear map from $X$ to $\Hom_{\F_q}(X,Y)$, we can also get at most five other rank metric codes by replacing the semifield multiplication $x\ast y$ with the bilinear form $b(\bx,\by)=\bx \beta(\by)$ over the subspace $X$. The precise approach can be found in \cite{edel_representations_2010}; see \cite{dempwolff_dimensional_2013} for the special case of bilinear DHOs with $X=Y$. For the convenience of the reader, we include some details here:

For $\cD:=\{\beta(\ba): \ba\in X \}\subseteq \F_q^{n\times r}$, we write
\[\mathbb{D}:=\{X({\bf a}) : {\bf a} \in X\},\]
where $X({\bf a})=\{(\bx,\bx\beta({\bf a})) : \bx\in X\}$ for ${\bf a} \in X$. 

Let
\[\mathcal D^{\circ} :=\{\beta^{\circ}({\bf a}) :{\bf a} \in X\}\text{ and }\mathbb{D}^{\circ}:=\{X^{\circ}({\bf a}) : {\bf a} \in X\},\]
where $X^{\circ}({\bf a})=\{(\bx,\bx\beta^{\circ}({\bf a})) : \bx\in X\}$ for ${\bf a} \in X$. 

In particular, when $\cD$ is a bilinear DHO-set, taking $\ba={\bf 0}$ in Condition \ref{item.P2}, we see that each $\bb$ is mapped bijectively to the unique nonzero element in $\ker(\beta(\bb))$, whence the rank of $\beta^\circ (\bb)$ is also $n-1$ for each $\bb\in X\setminus\{\bf 0 \}$, i.e., Condition \ref{item.P1} is satisfied for $\bD^\circ$. It is straightforward to verify that Condition \ref{item.P2} also holds. Therefore we have proved the following result.
\begin{lemma}\label{lm:opposite_DHO}
	Let $\bD$ be a bilinear DHO. Then $\bD^\circ$ is a bilinear DHO as well.
\end{lemma}

On the other hand, we have another map $\beta^{\dagger}$ from $X$ to $\Hom_{\F_q}(Y,X)$. Precisely, fix two non-degenerate symmetric bilinear forms $\bili_X(\cdot,\cdot)$ on the subspace $X$ and $\bili_Y(\cdot,\cdot)$ on $Y$, and for each ${\bf a} \in X$ and consider the adjoint $\beta^{\dagger}({\bf a})$ of $\beta({\bf a})$ with respect to $\bili_X$ and $\bili_Y$.
We set 
$${\mathcal D}^{\dagger}:=\{\beta^{\dagger}(\ba) : \ba \in X\} \text{ and } {\mathbb D}^{\dagger}:=\{X^{\dagger}({\bf a}) : {\bf a} \in X\},$$ 
where $X^{\dagger}({\bf a})=\{(\by,\by\beta^{\dagger}({\bf a})) : {\bf a} \in X\}\subseteq Y\oplus X$. Since, as observed before, we can choose appropriate $\F_q$-bases of $X$ and $Y$ in such a way that ${\mathcal D}^{\dagger} = {\mathcal D}^{\top}$, we simply denote $\mathbb{D}^{\dagger}$ and ${\mathcal D}^{\dagger}$ by $\mathbb{D}^\top$ and ${\mathcal D}^{\top}$ respectively in the rest of this paper.

In particular, when $\cD$ is a bilinear DHO-set, Condition \ref{item.P1}, i.e., $\dim(\ker(\beta^{\dagger}({\bf a}))=1$ for every element ${\bf a}\in X \setminus \{0\}$, is satisfied. However, Condition \ref{item.P2}, is not satisfied in general. 

Summing up, starting from $\bD$ or $\cD$, and using the opposite operation $\circ$ and the adjoint operation $\top$, we obtain up to six objects ${\bD}, {\bD}^{\circ}, {\bD}^{\top}, {\bD}^{\circ \top}, {\bD}^{\top \circ}$ and ${\bD}^{\circ \top\circ}={\bD}^{\top \circ \top}$ as well as at most six subspaces of bilinear forms ${\mathcal D}, {\mathcal D}^{\circ}, {\mathcal D}^{\top}, {\mathcal D}^{\circ \top}, {\mathcal D}^{\top \circ}$ and ${\mathcal D}^{\circ \top\circ}={\mathcal D}^{\top \circ \top}$. We call each element in $\{\id, \circ, \top, \circ\top, \top\circ, \circ\top\circ, \top\circ\top \}$ a \emph{Knuth operation}.

In particular, when $\bD$ is a bilinear DHO of rank $n$ in $\F_2^{2n}$ which splits over one of its elements, as pointed by Edel in \cite{edel_representations_2010}, all these six objects $\bD^k$ for a Knuth operation $k$ are bilinear DHOs if $\bD^\top$ is a DHO. Moreover,  $\bD^\top$ is a DHO if and only if $\bD$ is \emph{doubly dual}, i.e., $X_1+X_2$ has codimension $1$ in $U$ and $X_1+X_2+X_3=U$ for three different $X_1,X_2,X_3\in \bD$, where $U$ is the ambient space of $\bD$; see \cite{dempwolff_symmetric_2015,taniguchi_duals_2009}.

\begin{remark}\label{remark:alternating_symmetric_DHO}
	A DHO $\bD$ is \emph{symmetric}, if $\bD^\circ =\bD$, i.e., if $\mathbb{D}$ is determined by an injective $\F_q$-linear map $\beta \colon X \rightarrow \Hom(X,Y)$ such that $(\bx)\beta(\ba)=(\ba)\beta(\bx)$ for all $\bx,\ba \in X$, where $X\cong \F_2^n$ and $Y\cong \F_2^r$. A DHO $\bD$ is \emph{alternating}, if $\ba \beta (\ba)=\bf 0$ for each $\ba\in X$. It is not difficult to verify that an alternating dual hyperoval is symmetric. 	In \cite[Theorem 2.4]{dempwolff_dimensional_2014}(partial results can also be found in \cite{edel_quadratic_2010,yoshiara_notes_2010}), Dempwolff and Edel proved that an alternating DHO determined by a monomorphism $\beta : X\rightarrow \Hom(X,Y)$ where $X\cong \F_2^n$ and $Y\cong \F_2^r$ is equivalent to a quadratic APN function from $\F_2^n$ to $\F_2^r$. A function $f: \F_2^n\rightarrow \F_2^r$ is called \emph{almost perfect nonlinear} or \emph{APN function} for short, if it satisfies that for any $\ba\in X\setminus \{\bf 0\}$ and $\bb \in Y$ 
 
 the equation 
	\[f(\bx + \ba) + f(\bx) = \bb\]
	has at most two solutions. A function  $f: \F_2^n\rightarrow \F_2^r$ is called \emph{quadratic} if the map from $\F_2^n\times \F_2^n$ to $\F_2^r$ defined by
	\[ (\bx, \by) \mapsto f(\bx + \by) + f(\bx) + f(\by) \]
	is bilinear. APN functions have the optimal properties for offering resistance against differential cryptanalysis and they have been intensively studied by many mathematicians. For recent surveys on APN functions, we refer to \cite{blondeau_perfect_2015,pott_almost_2016}.
\end{remark}

Next we consider the links among the kernels and the nuclei of ${\mathcal D}, {\mathcal D}^{\circ}, {\mathcal D}^{\top}, {\mathcal D}^{\circ \top}, {\mathcal D}^{\top \circ}$ and ${\mathcal D}^{\circ \top\circ}={\mathcal D}^{\top \circ \top}$. Similar results for the kernels and the nuclei of semifields are well known, see \cite{lavrauw_semifields_2011,marino_nuclei_2012}.

\begin{lemma}\label{lm:kernel_nuclei_knuth}
	Let $\C$ be an $\F_q$-linear subset of size $q^n$ in $\F_q^{n\times r}$. Let $K(\C^\circ)$ denote the kernel of the translation structure $\scrT(\C^\circ)$ associated with $\C^\circ$. If at least one of $N_r(\C)$ and $ K(\C^\circ)$ is a field, then $N_r(\C) \cong K(\C^\circ)$.
\end{lemma}
\begin{proof}
	As $\C$ is $\F_q$-linear and $\#\C=q^n$, $\C$ forms an $n$-dimensional vector space over $\F_q$. Thus there exists an $\F_q$-linear injection $\beta:\F_q^n \rightarrow \F_q^{n\times r}$ such that
	\[ \C=\{ \beta(\ba): \ba \in \F_q^n \}.\]
	Let $Z$ be in $N_r(\C)$. According to definition, for any $\by\in \F_q^n$,
	$\beta(\by) Z\in \C$. It means that there exists a map $\zeta$ from $\F_q^n$ to itself such that $\beta(\by)Z=\beta(\zeta(\by))$. Moreover, it is straightforward to verify that $\zeta$ is also $\F_q$-linear, which implies that $\zeta$ corresponds to a matrix $N_Z\in \F_q^{n\times n}$. By calculation, for any $\bx, \by\in \F_q^n$, we have
	\[(\by, \by \beta^\circ (\bx))\left(
	 \begin{array}{cc}
		N_Z & O \\
		O & Z \\
	 \end{array}
	\right)=(\by N_Z, \by \beta^\circ (\bx) Z)
	=(\by N_Z, \bx \beta (\by) Z)
	=(\by N_Z, \bx \beta (\by N_Z)),
	\]
	which equals $(\by N_Z, \by N_Z \beta^\circ (\bx))\in \C^\circ$. Hence the matrix
	\[
	\left(
	 \begin{array}{cc}
		N_Z & O \\
		O & Z \\
	 \end{array}
	\right)
	\]
	is in $K(\C^\circ)$. We also have to prove that this matrix is uniquely determined by $Z$: 
	\begin{itemize}
	\item	Assume that $N_r(\C)$ is a field. It means that $Z$ is of full rank and $\zeta$ is a bijection, which implies that $N_Z$ is also of full rank and uniquely determined by $Z$.
	\item  Assume that $ K(\C^\circ)$ is a field. By Lemma \ref{lm:kernel_N1_N2}, $Z$ and $N_Z$ are both invertible. Hence $N_Z$ is uniquely determined by $Z$ as well.
	\end{itemize}
	
	Now let us show that every element in $ K(\C^\circ)$ corresponds to a unique element in  $N_r(\C)$. Let
	\[
	\left(
	 \begin{array}{cc}
		N_1 & O \\
		O & N_2 \\
	 \end{array}
	\right)
	\]
	be an arbitrary element in $K(\C^\circ)$. Then it is straightforward to get
	\[ (\by N_1, \by N_1 \beta^\circ(\bx)) = (\by N_1, \by \beta^{\circ}(\bx) N_2 )\]
	for all $\bx,\by\in \F_q^n$, from which it follows that
	\[\bx \beta(\by N_1)=\bx\beta(\by)N_2.\]
	Thus $N_2$ is in $N_r(\C)$. Under the assumption that at least one of $N_r(\C)$ and $ K(\C^\circ)$ is a field, we can show that $N_2$ is invertible from which it follows that $N_1$ is uniquely determined by $N_2$.
	
	Therefore, $N_r(\C) \cong K(\C^\circ)$.
\end{proof}

By Lemma \ref{lm:kernel_nuclei_knuth} and Proposition \ref{prop:Delsartedual}, we obtain the following results.
\begin{theorem}\label{th:six_kernel_nuclei}
	Let $\C$ be an $\F_q$-linear subset of size $q^n$ in $\F_q^{n\times r}$. Assume that the kernels of the translation structures associated with $\C$, $\C^\circ$, $\C^\top$, $\C^{\circ\top}$, $\C^{\top\circ}$ and $\C^{\top\circ\top}$ are all fields. Then we have
	\begin{enumerate}[label=(\alph*)]
	\item\label{item:six_1} $N_r(\C) \cong K(\C^\circ) \cong N_m(\C^\top)$;
	\item\label{item:six_2} $N_r(\C^\circ) \cong K(\C) \cong N_m(\C^{\circ\top})$;
	\item\label{item:six_3} $N_r(\C^{\top \circ}) \cong K(\C^\top) \cong N_m(\C^{\top\circ\top})$;
	\item\label{item:six_4} $N_r(\C^{\top \circ \top}) \cong K(\C^{\circ \top}) \cong N_m(\C^{\top\circ})$;
	\item\label{item:six_5} $N_r(\C^\top) \cong K(\C^{\top\circ}) \cong N_m(\C)$;
	\item\label{item:six_6} $N_r(\C^{\circ \top}) \cong K(\C^{\top\circ\top}) \cong N_m(\C^\circ)$.
	\end{enumerate}
\end{theorem}

By Proposition \ref{prop:DHOkernel}, we easily get the following result.
\begin{corollary}\label{coro:right_nucleus_DHO}
	Let $\bD$ be a bilinear DHO of rank $n$ with ambient space of dimension $n+r$ over $\F_2$. The right nucleus of the DHO-set associated with $\mathbb{D}$ is $\F_2$. 
\end{corollary}

Regarding the middle nucleus of a bilinear DHO-set, in \cite{dempwolff_dimensional_2014}, the following result was proven.
\begin{proposition}{\cite[Proposition 3.9(b)]{dempwolff_dimensional_2014}}\label{th:middle_nucleus_DHO}
	Let $\cD$ be the associated DHO-set of a bilinear DHO of rank $n$ with $n>2$. Then there exists a positive integer $\ell$ dividing $n$ in such a way that the middle nucleus of $\cD$ is isomorphic to $\F_{2^\ell}$.  
\end{proposition}

About the theorem above, we warn the reader that in \cite{dempwolff_dimensional_2014} the middle nucleus is called the nucleus of the DHO. Also in \cite{dempwolff_dimensional_2014}, the following results are obtained. For $r=n-1$, projections of spreads of commutative semifields provide examples with various sizes of middle nuclei, see \cite[Example 6.3]{dempwolff_dimensional_2014}. Furthermore, when $\cD$ is alternating,  the elements in $N_m(\cD)$ must be in a special form and $N_m(\cD)$ is isomorphic to $\F_2$ or $\F_4$. If the second case occurs, then $n$ must be even. See \cite[Proposition 3.9(f)]{dempwolff_dimensional_2014}.

\medskip

In the final part, let us concentrate on the case that $\cD$ is a DHO-set in $\F_2^{n\times n}$ associated with a bilinear DHO. From Proposition \ref{prop:DHOkernel} and Corollary \ref{coro:right_nucleus_DHO}, we see that the kernels and the nuclei $K(\cD^{\circ})\cong N_r(\cD) \cong N_m(\cD^{\top})$ and $K(\cD) \cong N_r(\cD^{\circ}) \cong N_m(\cD^{\circ \top})$ in case $(a)$ and $(b)$ in Theorem \ref{th:six_kernel_nuclei} are all isomorphic to $\F_2$. By Theorem \ref{th:middle_nucleus_DHO}, the kernels and the nuclei $K(\cD^{\top \circ}) \cong N_r(\cD^{\top}) \cong N_m(\cD)$ and $K(\cD^{\top \circ \top}) \cong N_r(\cD^{\circ \top}) \cong N_m(\cD^{\circ})$ in \ref{item:six_5} and \ref{item:six_6} are all isomorphic to finite fields containing $\F_2$. By duality, the  same result holds true for kernels and nuclei\ref{item:six_3} and \ref{item:six_4}. Indeed, we can prove the following more general result.

\begin{lemma}\label{lm:oppositetranspose}
Let $\bD$ be a DHO of rank $n\geq 3$ with ambient space $V=\F_q^{2n}$. Let $\sigma(\cdot, \cdot)$ be a non-degenerate symmetric bilinear form on $V$ and set ${\bD}^{\dagger}=\{X^{\dagger} \,:\, X \in \bD\},$ where $X^{\dagger}=\{v\in V \,:\, \sigma(x,v)=0, \,x\in X\}$. Then,
$K({\bD}^{\dagger})\simeq \F_q$. 
\end{lemma}
\begin{proof}
Clearly, $K=\{\omega 1_V \,\colon \, \omega \in \F_q\}$ lies in $K(\bD^{\dagger})$. Let $\epsilon$ be an element of $K(\bD^{\dagger})$. As $\bD$ is a DHO, we see that for each $X\in \mathbb D$ and each point $P\in X$, there exist a unique $X' \in \bD$ with $X \cap X' = P$. So for each $X\in {\bD}^{\dagger}$ and each hyperplane $H \subset V$ of $V$ such that $X\subset H$, there exist a unique $X' \in \bD^{\dagger}$ such that $X+X'=H$. Therefore, $\epsilon$ fixes each hyperplane of $V/X$ and hence each point of this space. By the fundamental theorem of projective geometry $\epsilon$ induces $\mu 1_{V/X}$ on $V/X$ for some $\mu \in \F_q$. Similarly, if we take $X'\in \bD^{\dagger}\setminus \{X\}$, then $\epsilon$ induces $\mu'  1_{V/X'}$ on $V/X',$ for some $\mu' \in \F_q$. 

Let $v\in V \setminus \{X+X'\}$. Then, $v^{\epsilon} = \mu v +x = {\mu}' v +x'$, with $x\in X$ and $x' \in X'$. Hence, $(\mu - \mu')v \in X+X',$ i.e., $\mu = \mu'$. So $\epsilon$ induces $\mu 1_{V/(X \cap X')}$ on $V/(X \cap X')$. As $V=\langle \bD \rangle$, we have $\bigcap_{X\in \bD^{\dagger}}=0$. This forces $\epsilon = \mu 1_V$.
 \end{proof}

Let $V=\F_2^n \times \F_2^n$. As observed in Section \ref{Bilinear DHOs, their kernels and nuclei}, we may set $\sigma((x,y), (x',y'))=xy'+ yx'$. It is then easy to see that the adjoint operation on $\cD$ with respect to $\sigma$ is exactly $\top$. Hence, as a direct consequence of Lemma \ref{lm:oppositetranspose}, we have the following.

\begin{proposition}
	Let $\cD$ be the DHO-set associated with a bilinear DHO $\bD$ of rank $n$ in the ambient space of dimension $2n$, where $n>2$.  Then the kernel of $\cD^k$ is isomorphic to $\F_2$ for any Knuth operation $k\in \{\top, \circ\top\}$.
\end{proposition}

\section*{Acknowledgment}
The authors are grateful to the two anonymous referees for their valuable suggestions and comments. This work is supported by the Research Project of MIUR (Italian Office for University and Research) ``Strutture geometriche, Combinatoria e loro Applicazioni" 2012. Yue Zhou is supported by the Alexander von Humboldt Foundation and the National Natural Science Foundation of China (No.\ 11401579, 11531002).


\end{document}